\documentclass[twoside,11pt]{article}
\usepackage[preprint]{jmlr2e}
\makeatletter
\let\proof\relax
\@for\jmlr@env:=example,theorem,lemma,proposition,remark,corollary,definition,conjecture,axiom\do{%
  \expandafter\let\csname\jmlr@env\endcsname\relax
  \expandafter\let\csname end\jmlr@env\endcsname\relax}
\makeatother
\usepackage[utf8]{inputenc}
\usepackage[T1]{fontenc}
\usepackage{url}
\usepackage{booktabs}
\usepackage{microtype}
\usepackage{xcolor}
\usepackage{needspace}

\usepackage{amssymb,amsmath,amsthm,mathrsfs,mathtools,bm}

\usepackage{graphicx}
\graphicspath{{./}}
\usepackage{tabularx}
\usepackage{array}
\newcolumntype{L}[1]{>{\raggedright\arraybackslash}p{#1}}
\usepackage{enumitem}
\def\R{{\mathbb R}}
\def\M{{\mathcal M}}
\def\P{{\mathscr P}}
\def\X{{\mathcal X}}
\def\V{{\mathcal V}}
\def\W{{\mathcal W}}

\def\S{{\mathcal S}}

\def\E{{\mathbb E}}
\def\N{{\mathbb N}}

\def\Tr{\mathop{\rm Tr}}

\def\Var{{\rm Var}}
\def\dist{{\rm dist}}

\def\Prob{{\mathbb P}}

\newtheorem{thm}{Theorem}[section]
\newtheorem{proposition}[thm]{Proposition}
\newtheorem{lemma}[thm]{Lemma}
\newtheorem{corollary}[thm]{Corollary}
\newtheorem{theorem}[thm]{Theorem}
\newtheorem{assumption}[thm]{Assumption}
\theoremstyle{definition}
\newtheorem{definition}[thm]{Definition}

\newtheorem{remark}[thm]{Remark}

\ShortHeadings{Operator Calculus for Population-Based Optimizers}{Malo, Viitasaari, Nummi, Suominen, Sinha, and Tahvonen}
\firstpageno{1}

\begin{document}

\title{Operator Calculus for Population-Based Optimization:\\
Modular Convergence and Finite-Population Guarantees}

\author{\name Pekka Malo \\
       \name Lauri Viitasaari \\
       \addr Department of Information and Service Management\\
       Aalto University, Finland
       \AND
       \name Patrik Nummi \\
       \addr Department of Information and Communications Engineering\\
       Aalto University, Finland
       \AND
       \name Antti Suominen \\
       \addr Department of Information and Service Management\\
       Aalto University, Finland
       \AND
       \name Ankur Sinha \\
       \addr Operations and Decision Sciences\\
       Indian Institute of Management Ahmedabad, India
       \AND
       \name Olli Tahvonen \\
       \addr Department of Economics and Management\\
       University of Helsinki, Finland}

\maketitle

\begin{abstract}
Population-based optimizers combine update rules such as mutation, selection, and recombination. When one rule changes, it is often unclear which convergence guarantees survive or how the new combination should be assessed. We develop an operator calculus: an operator is a population-update rule, and the calculus specifies how separately checked effects can be combined. Under explicit regularity and small-step conditions, the leading changes caused by the updates add, yielding reusable building blocks for convergence analysis. The framework distinguishes finding and retaining a good solution, reducing the population's mean objective, and concentrating candidates near an optimizer, and identifies the extra approximation conditions needed for finite evaluation-budget guarantees. Applications include distribution adaptation, recombinative evolution, and consensus dynamics, with verified nonconvex cases. Controlled experiments on a common nonconvex problem collection show how component effects change with population geometry and the performance measure: a rule can worsen the mean objective yet produce better candidates.
\end{abstract}


\section{Introduction}\label{sec:intro}
Population-based optimizers maintain a set of candidate solutions, or the parameters of a distribution that generates candidates, and improve it by repeatedly applying update rules: perturbing candidates, favoring the better ones, and combining information from several of them. Evolution strategies, covariance-matrix adaptation (CMA-ES), and consensus-based optimization (CBO) illustrate these structures~\citep{beyerEvolutionStrategiesComprehensive2002,hansenCMAEvolutionStrategy2023,pinnauConsensusbasedModelGlobal2017}, as do natural evolution strategies and information-geometric optimization, which update a parametric search distribution~\citep{wierstraNaturalEvolutionStrategies2014,ollivier2017igo}. These methods use objective values only and therefore apply to black-box problems.

Population-based optimization is also a continuing algorithm-design task: practitioners combine search, selection, and information-exchange rules for the problems they face. We call each population-update rule an \emph{operator}; our \emph{calculus} provides rules for checking and combining their effects. For a new combination, a theorem for one fixed algorithm does not answer two questions: are the replacement updates compatible with the same population evolution, and does their combined effect improve the chosen error measure? A third question concerns algorithms that already work. When an algorithm is evaluated only as a whole, the evaluation can show that it performs well but not which of its ingredients is responsible, or what adding an operator to an existing method contributes. Because the effects of separately specified operators add to first order, the calculus attributes the change of an error measure at a given population to the individual operators, and this attribution depends on the problem and on the geometry of the population.

\begin{figure}[t]
\centering
\includegraphics[width=\textwidth]{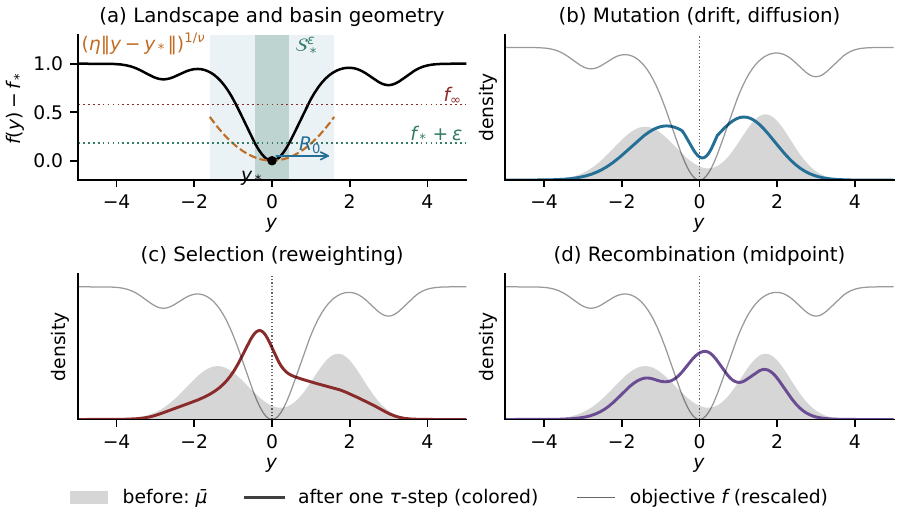}
\caption{Landscape conditions and canonical operators on one nonconvex objective. (a) The geometry of Assumption~\ref{assumption:basin}: inside the basin $B_{R_0}(y_*)$ (blue) the objective gap dominates the sharpness envelope $(\eta\|y-y_*\|)^{1/\nu}$ (dashed), so small gaps imply proximity to $y_*$; outside it the gap exceeds $f_\infty$ (dotted), which excludes near-optimal points at the local minima. The $\varepsilon$-optimal set $\S_*^\varepsilon$ (green) is the target of the discovery criterion. Assumption~\ref{assumption:landscape} controls regularity and growth of $f$. (b)--(d) One balanced $\tau$-step of each canonical operator of Section~\ref{subsec:canonical-examples} on the same population density (grey; thin line: $f$): mutation moves and spreads mass, selection reweights it without creating states, and midpoint recombination creates offspring between parents.}
\label{fig:landscape-operators}
\end{figure}

Related work provides distributional constructions~\citep{ollivier2017igo}, kinetic descriptions~\citep{borghi2023kinetic,borghi2024unified}, and particle approximations~\citep{delmoral2001genetic,borghi2026chaos} of such algorithms. Our focus is an \emph{operator-wise verification calculus}: which conditions permit separately specified updates to be composed, how do their dissipation bounds combine, and which checks remain valid when a component changes? Section~\ref{app:related-work} gives the detailed comparison.

The general results concern population distributions and updates satisfying stated moment, continuity, and small-step conditions, without prescribing their mechanisms. Each update's leading effect on a measured population quantity is described by its \emph{generator}; these effects add to define a continuous-time population evolution, and mutation, selection, and recombination yield a transport--reaction--jump (TRJ) equation. A sampling kernel separates internal states from evaluated candidates, allowing a state-space error estimate to control candidate quality or concentration even when the state consists of distribution parameters.

\paragraph{Contributions.}
\begin{enumerate}[label=\textup{(\arabic*)},leftmargin=*,itemsep=2pt]
\item \emph{Composition calculus} (Section~\ref{sec:operators}). Conditions on the individual updates under which the first-order effects of separately specified operators add (Theorem~\ref{thm:operator-composition}), in every order for which the conditions hold. For mutation, selection, and recombination the sum is the TRJ equation, under the coefficient assumptions of Section~\ref{subsec:canonical-examples} (Theorem~\ref{thm:TRJ-main}).
\item \emph{Lyapunov principle} (Section~\ref{sec:lyapunov}). A convergence theorem for the composed evolution with unbounded Lyapunov functions (Theorem~\ref{thm:mean-field-convergence}), in which the dissipation rate is assembled from component-wise bounds (Proposition~\ref{prop:operatorwise}) and the sampling kernel transfers a state-space estimate to the objective values, concentration, or consensus of the evaluated candidates.
\item \emph{Finite evaluation budgets} (Section~\ref{sec:finite-budget}). A criterion that converts mean-field concentration into a probability bound for finding an $\varepsilon$-optimal candidate within a stated number of evaluations (Corollary~\ref{thm:complexity-main}); it identifies the state-law approximation estimate as the algorithm-specific input, and a measured variant needs no mean-field approximation.
\item \emph{Demonstrations} (Sections~\ref{sec:numerics}--\ref{sec:verified}). A controlled study of eight assemblies of checked operators on ten nonconvex problem instances, with the operators switched on and off individually so that effects are attributed component by component, tests the additive predictions, the dependence of dissipation on population geometry, and finite-population discovery; it is meant to inform assembly design rather than promote an optimizer. Three verified instances, a CMA-ES-type parameter flow, a derivative-free recombinative evolution strategy, and CBO, show what a complete certificate requires.
\end{enumerate}

Separate component bounds can be conservative; their benefit is reuse, not uniformly sharp rates. An exploratory component may increase mean error yet produce a better candidate; whether this helps depends on the target: one retained solution, a reliable population, or agreement near an optimizer.

\paragraph{Organization and Online Appendix.}
Section~\ref{app:related-work} places the framework in the literature, and Section~\ref{sec:setup} states the problem, the landscape conditions, and the convergence targets. Sections~\ref{sec:operators}--\ref{sec:finite-budget} develop the calculus, the Lyapunov principle, and the finite-budget criterion, with notation and proofs in Appendices~\ref{app:prelim}--\ref{app:finite-population}. Sections~\ref{sec:numerics}--\ref{sec:verified} report the numerical study and the verified instances, and Section~\ref{sec:conclusion} discusses limitations. The Online Appendix, whose items are numbered with the prefix S, contains the complete numerical study (Section~S1), a practitioner's guide to the evaluation criterion, the measured discovery bound, and the estimation of drift coefficients (Section~S2), the verification of the canonical operator classes (Section~S3), the statements and proofs for the verified instances (Section~S4), sufficient conditions for the finite-budget hypotheses (Section~S5), and the code, frozen protocols, and results that reproduce every figure and table (Section~S7). The canonical formulas are examples, not premises of the general framework.

\section{Related Work and Scope of the Comparison}\label{app:related-work}

\subsection{Closest distributional and particle-system frameworks}\label{subsec:rw-convergence}
Measure-valued descriptions of evolutionary algorithms are well established. \citet{delmoral2001genetic} analyze genetic algorithms as interacting particle systems, with finite-population approximation results and crossover constructions, including a product-state representation in Section~7.3. Our state augmentation and finite-population perspective should not be interpreted as originating those ideas.

Information-geometric optimization (IGO)~\citep{ollivier2017igo} defines a natural-gradient flow on a parametric family of search distributions (Definition~4), proves consistency of its sampled update as the sample size grows (Theorem~6), and recovers a rank-$\mu$ CMA update (Proposition~15). That specialization omits evolution paths and step-size control, as explained in Section~5.2.1. The distinction between a parameter state and sampled candidates is therefore already present in IGO. Here the kernel lift serves a different purpose: transferring a Lyapunov certificate on possibly augmented states to objective and concentration estimates for generated candidates, while keeping separately represented operators additive. Example~S4.5 retains path and scale variables, but is only a one-dimensional expected-flow result.

\citet{borghi2023kinetic} derive a kinetic model incorporating selection, crossover, and mutation. Their Theorem~4.1 establishes convergence of the population mean under a landscape condition and sufficiently strong Boltzmann selection; Remark~4.1 distinguishes that observable from population concentration in the presence of nonvanishing mutation variance. \citet{borghi2024unified} review a kinetic unification of simulated annealing, genetic algorithms, PSO, and ensemble Kalman methods. We do not claim the first common distributional language for these algorithms. Our emphasis is the separation between representation, component-wise dissipation, and a finite-implementation guarantee.

For a specified crossover--mutation GA with bounded, strictly positive Lipschitz fitness, \citet[Theorem~1.3]{borghi2026chaos} gives a bounded-Lipschitz approximation estimate with separate population-size and time-step terms. Such results are the input needed by Corollary~\ref{thm:complexity-main}; to supply our $W_2$-based modulus from them, the metric, moments, update rule, and discretization must be matched, as Online Appendix~S2 spells out, and approximation for that GA must be redone for the augmented CMA-ES-type state.

\subsection{Algorithmic scope}\label{subsec:rw-algorithms}\label{subsec:rw-synthesis}
Classical evolutionary theory includes Markov-chain analyses~\citep{rudolphConvergenceAnalysisCanonical1994,voseSimpleGeneticAlgorithm1998}, progress-rate analyses~\citep{beyerEvolutionStrategiesComprehensive2002}, and drift or level-based bounds~\citep{corusLevelBasedAnalysisGenetic2018,doerr2020complexity}. Our direct recursion is closer to the last perspective than to propagation of chaos. CBO has its own mean-field convergence theory~\citep{pinnauConsensusbasedModelGlobal2017,carrilloAnalyticalFrameworkConsensusbased2018,fornasierConsensusBasedOptimizationMethods2024}; our localized certificate does not replace those algorithm-specific results.

The descriptions of DE~\citep{storn1997differential}, PSO~\citep{kennedy1995particle}, CMA-ES~\citep{hansenCMAEvolutionStrategy2023}, NES~\citep{wierstraNaturalEvolutionStrategies2014}, and EDAs~\citep{hauschild2011eda} motivate structural correspondences with the present framework (Table~S13); they do not themselves verify our regularity or dissipation hypotheses. Gradient methods~\citep{robbins1951stochastic,kingma2015adam,reddi2018convergence} likewise require their own noise and stability analyses. Hard ranks, truncation, greedy replacement, and memory rules may require smoothing or a separate direct generator argument. The claimed generality is modularity across admissible components, not universal convergence across these families.

\subsection{Mathematical infrastructure}\label{subsec:rw-math}
Metric gradient flows~\citep{ambrosioGradientFlowsMetric2005,jordanVariationalFormulationFokkerPlanck1998} and unbalanced transport~\citep{chizatInterpolatingDistanceOptimal2018,lieroOptimalEntropyTransportProblems2018} provide context, not a claimed variational representation of our transport--reaction--jump equation. McKean--Vlasov and propagation-of-chaos theory~\citep{chaintron2022chaos,carmonaProbabilisticTheoryMeanII2018} explain particle limits for the updates they cover; Online Appendix~S2 shows how such estimates enter the evaluation criterion. The composition argument is a nonlinear first-order consistency analogue of operator splitting~\citep{trotterApproximationSemiGroups1959,strang1968construction}; it neither proves convergence of repeated finite steps nor requires commutativity.

\section{Problem Setup and Convergence Modes}\label{sec:setup}

\paragraph{Notation.}
We write $\|\cdot\|$ for the Euclidean norm on $\R^n$ (dimension clear from context), $B_r(y):=\{z:\|z-y\|<r\}$ for the open ball, $\mathcal{B}(\S)$ for the Borel $\sigma$-algebra on a subset $\S\subseteq\R^d$, and $\delta_x$ for the Dirac mass at~$x$. For a measure~$\mu$ and a measurable function~$\varphi$, the integral pairing is $\langle\varphi,\mu\rangle:=\int\varphi\,d\mu$. The space $C_b^k(\X)$ consists of $k$-times continuously differentiable functions with bounded derivatives up to order~$k$. We denote by $\P_q(\X)$ the set of Borel probability measures on~$\X$ with finite $q$-th moment, i.e.\ $\int\|x\|^q\,d\mu<\infty$, and write $[a]_+:=\max(a,0)$. For $p\ge1$ and $\mu,\nu\in\P_p(\X)$, $W_p(\mu,\nu):=\inf_{\gamma\in\Gamma(\mu,\nu)}(\int\|x-y\|^p\,d\gamma)^{1/p}$ is the $p$-Wasserstein distance, with $\Gamma(\mu,\nu)$ the set of couplings of $\mu$ and $\nu$. Appendix~\ref{app:prelim} collects the further notation used in the proofs.

We consider the problem of minimizing an objective $f\colon\S\to\R\cup\{+\infty\}$ over a search space $\S\subset\R^d$, where $f$ may be nonconvex and nonsmooth. We write $f_*:=\inf_{y\in\S} f(y)$ for the global infimum, $\S_*:=\{y\in\S\colon f(y)=f_*\}$ for the set of global minimizers, and $\S_*^\varepsilon:=\{y\in\S\colon f(y)\le f_*+\varepsilon\}$ for the $\varepsilon$-optimal sublevel set. Two landscape conditions, used only where explicitly invoked, are the following; Figure~\ref{fig:landscape-operators}(a) illustrates the second.

\begin{assumption}[Fitness landscape regularity]\label{assumption:landscape}
There exist $L_f,c_\ell,c_u>0$, $s\ge0$, and $u\ge\ell>0$ such that for all $y,y'\in\S$:
\textup{(i)} $|f(y)-f(y')|\le L_f(1+\|y\|+\|y'\|)^s\|y-y'\|$;
\textup{(ii)} $f_*>-\infty$ and $c_\ell(\|y\|^\ell-1)\le f(y)-f_*\le c_u(1+\|y\|^u)$.
\end{assumption}

\begin{assumption}[Basin regularity]\label{assumption:basin}
There exist $y_*\in\S_*$, a basin radius $R_0>0$, a sharpness exponent $\nu\in(0,\infty)$, and $\eta,f_\infty>0$ such that:
\textup{(i)} $\|y-y_*\|\le\eta^{-1}[f(y)-f_*]^\nu$ for all $y\in B_{R_0}(y_*)$;
\textup{(ii)} $f(y)-f_*>f_\infty$ for all $y\notin B_{R_0}(y_*)$.
\end{assumption}

Assumption~\ref{assumption:landscape} extends the regularity and growth conditions of CBO analyses~\citep{carrilloAnalyticalFrameworkConsensusbased2018} to polynomial growth: $s=0$ is global Lipschitz continuity, larger $s$ covers quartic double-well landscapes, and (ii) makes sublevel sets bounded. Assumption~\ref{assumption:basin} appears in~\citet{fornasierConsensusBasedOptimizationMethods2024}; $\nu=1/2$ corresponds to quadratic growth and smaller $\nu$ to flatter growth. Both describe geometry, not convexity or dynamical invariance. In Figure~\ref{fig:landscape-operators}(a), inside the basin the gap also lies below the local Lipschitz bound $L_0\|y-y_*\|$ of Assumption~\ref{assumption:landscape}(i). Sharpness converts small objective gaps within $B_{R_0}(y_*)$ into proximity to $y_*$, while the strict exterior gap places $\S_*^\varepsilon$ inside this basin whenever $\varepsilon<f_\infty$. Neither condition establishes that an algorithm enters or remains in the basin. Such dynamical conclusions require the separate drift and invariance checks stated in the relevant results.

\paragraph{State space versus search space.}
The algorithm maintains a population described by a (possibly unnormalized) measure $\mu_t$ on an internal state space~$\X\subseteq\R^{d_x}$, which may differ from the search space~$\S$ where $f$ is evaluated (in general $d_x\ne d$). We fix a moment order $q\ge2$ and write $\M_{q,+}(\X)$ for the cone of nonnegative Borel measures on $\X$ with finite $q$-th moment and positive total mass. The normalization $\bar\mu_t:=\mu_t/\mu_t(\X)$ is a probability measure representing the population distribution. A \emph{sampling kernel} $K\colon\X\times\mathcal{B}(\S)\to[0,1]$ maps each internal state $x\in\X$ to a probability distribution $K(x,\cdot)$ on $\S$, and the \emph{induced search law} is
\begin{equation}\label{eq:search-law}
\nu_t := \Pi[\bar\mu_t], \qquad \Pi[\bar\mu](A) := \int_\X K(x,A)\,\bar\mu(dx).
\end{equation}
When states are candidates, $\X=\S$ and $K(x,\cdot)=\delta_x$, so $\nu_t=\bar\mu_t$. For parametric methods such as CMA-ES, $x$ encodes a Gaussian's mean $m(x)$, covariance shape $C(x)$, and scale $\sigma(x)$, with $K(x,\cdot)=\mathcal{N}(m(x),\sigma(x)^2 C(x))$. Thus internal states and evaluated candidates need not coincide. The one-sample success probability $p_\varepsilon(t):=\nu_t(\S_*^\varepsilon)$ measures the chance of drawing an $\varepsilon$-optimal candidate.

\begin{definition}[Mean-field convergence modes]\label{def:mf-modes}
The mean-field curve achieves:
\begin{enumerate}[label=\textup{(MF\arabic*)},leftmargin=*,itemsep=2pt]
\item $\varepsilon$-concentration if $p_\varepsilon(t)\to1$ for every $\varepsilon>0$;
\item population mean-objective convergence if $\int f\,d\nu_t\to f_*$;
\item optimal consensus if $\S_*=\{y_*\}$ and $W_2(\nu_t,\delta_{y_*})\to0$, with $W_2$ the $2$-Wasserstein distance.
\end{enumerate}
\end{definition}

MF2 implies MF1 by Markov's inequality. MF3 implies MF2 if $f$ is continuous at $y_*$ and uniformly integrable along $(\nu_t)$; $W_2$ alone does not control superquadratic objective tails. Assumption~\ref{assumption:basin} forces a unique minimizer; without it, objective convergence can approach several minimizers without consensus.

\subsection{Three targets: a retained solution, reliable candidates, agreement}\label{app:convergence-targets}
For $f_*\in\R$, let $\mathcal Y_{\le k}$ contain all candidates evaluated through generation $k$ and define the \emph{best-so-far gap} $B_k:=\min_{y\in\mathcal Y_{\le k}}[f(y)-f_*]$, the objective gap of the best individual evaluated so far. Then $B_k\downarrow0$ concerns the best value found along the path, not the current population: it can hold under continuing exploration without MF2 or MF3, and agreement at a suboptimal point is not optimal consensus.

Table~\ref{tab:convergence-targets} lists the three targets with the observable and the certificate of each: the best-so-far gap $B_k$, the mean gap $\int(f-f_*)\,d\nu$, and the squared distance to an optimizer. They are not ordered. For $\bar y_\nu:=\int y\,d\nu$ and $\Var(\nu):=\int\|y-\bar y_\nu\|^2\,d\nu$, expanding the square gives
\begin{equation}\label{eq:consensus-spread-bias}
W_2^2(\nu,\delta_{y_*})=\Var(\nu)+\|\bar y_\nu-y_*\|^2,
\end{equation}
so zero variance can mean agreement at the wrong point, while for the quartic $f(y)=(y^2-1)^2$ the law $(\delta_{-1}+\delta_1)/2$ has zero objective gap but positive variance and mean location $0$ with $f(0)=1$, the clustered-start geometry of the quartic experiments of Section~\ref{sec:numerics}. Conversely, mixing a fixed exploration law $Q$ with $Q(\S_*^\varepsilon)>0$ into the proposals with weights $\eta_k$ makes every batch miss with conditional probability at most $(1-\eta_kQ(\S_*^\varepsilon))^{N_k}$, so $\sum_kN_k\eta_k=\infty$ gives $B_k\to0$ almost surely while the mean gap and the variance may stay constant: best-so-far convergence can hold without MF2 or MF3.

\begin{table}[tbp]
\centering\small
\begin{tabularx}{\linewidth}{@{}p{0.17\linewidth}p{0.36\linewidth}X@{}}\toprule
Target & Observable; certificate & Not required or not guaranteed\\\midrule
One retained solution & $B_k=\min_{y\in\mathcal Y_{\le k}}(f(y)-f_*)$; discovery bound (Corollary~\ref{thm:complexity-main}) & Current population need not improve or agree.\\
Reliable current candidates & $\int(f-f_*)\,d\nu_t$; Theorem~\ref{thm:mean-field-convergence}(b) & Need not select one minimizer when several exist.\\
Agreement at an optimizer & $W_2^2(\nu_t,\delta_{y_*})$; Theorem~\ref{thm:mean-field-convergence}(c) & Small spread alone does not establish optimality.\\\bottomrule
\end{tabularx}
\caption{Targets, observables, and certificates. Mean objective means $\int f\,d\nu$, not $f(\int y\,d\nu)$.}\label{tab:convergence-targets}
\end{table}

\section{General Operator Calculus}\label{sec:operators}

Operators act on population measures $\mu\in\M_{q,+}(\X)$. Their pre-generators generalize infinitesimal generators~\citep{ethierMarkovProcessesCharacterization} to population-dependent updates, as in McKean--Vlasov models~\citep{ambrosioGradientFlowsMetric2005,carmonaProbabilisticTheoryMean2018}. Under the conditions below, these contributions add.

\subsection{Population metric}\label{subsec:pop-metric}

An algorithmic macro-step is a composition of $J\ge1$ operators, $T_\tau:=T_\tau^{(J)}\circ\cdots\circ T_\tau^{(1)}$, with $T_\tau^{(j)}\colon\M_{q,+}(\X)\to\M_{q,+}(\X)$ and step size $\tau>0$. No particular operator formula is imposed; admissibility is specified below.

Mutation and recombination move mass, whereas unbalanced selection also changes its total. To capture both, supplement the $2$-Wasserstein distance $W_2$~\citep{ambrosioGradientFlowsMetric2005,villani2009optimal} on normalized laws with the \emph{weighted bounded-Lipschitz distance}
\begin{equation}\label{eq:dBLq}
d_{\mathrm{BL},q}(\mu,\nu):=\sup\bigl\{|\langle\varphi,\mu-\nu\rangle|\colon\|\varphi\|_{\mathrm{BL},q}\le1\bigr\},
\end{equation}
with test-function norm
\begin{equation}\label{eq:BL-norm}
\|\varphi\|_{\mathrm{BL},q}
:=\sup_{x}\frac{|\varphi(x)|}{1+\|x\|^q}
+\sup_{x\neq y}\frac{|\varphi(x)-\varphi(y)|}{\|x-y\|\,(1+\|x\|^{q-1}+\|y\|^{q-1})}.
\end{equation}
This distance metrizes weak convergence together with convergence of $q$-th moments, and is well defined for measures of different total mass.

The \emph{population metric} combines both:
\begin{equation}\label{eq:pop-metric}
d_{q,2}^*(\mu,\nu):=d_{\mathrm{BL},q}(\mu,\nu)+W_2(\bar\mu,\bar\nu).
\end{equation}
The two terms control reweighting and motion, respectively.

\subsection{Pre-generators}\label{subsec:pregen-main}

For bounded test functions $\varphi\in\mathcal D:=C_b^3(\X)$, with topological dual $\mathcal D'$ (Appendix~\ref{app:prelim}), define the \emph{pre-generator}
\[
G_j[\mu](\varphi) \;:=\; \lim_{\tau\downarrow0}\frac{\langle\varphi,\,T_\tau^{(j)}[\mu]\rangle - \langle\varphi,\mu\rangle}{\tau},
\]
when the limit exists. It measures the instantaneous change of the population average and can depend on the whole law $\mu$, for example through a consensus point or selection pressure.

\begin{assumption}[Composition-admissible operator splitting]\label{assumption:composite}
For $r<\infty$ and $m_0>0$, put
\[
\M_{q,+}^{r,m_0}:=\{\mu:\ \mu(\X)\ge m_0,\ \int(1+\|x\|^q)\,d\mu\le r\}.
\]
Write $\mu^{(0)}=\mu$ and $\mu^{(j)}=T_\tau^{(j)}[\mu^{(j-1)}]$. Each component satisfies:
\begin{enumerate}[label=\textup{(A\arabic*)},leftmargin=*,itemsep=2pt]
\item \emph{Uniform pre-generator.} For each $\varphi\in\mathcal D$,
\[
\sup_{\mu\in\M_{q,+}^{r,m_0}}\left|
\frac{\langle\varphi,T_\tau^{(j)}[\mu]-\mu\rangle}{\tau}
-G_j[\mu](\varphi)\right|\longrightarrow0.
\]
\item \emph{Moment and mass stability.} For all $\tau\in(0,\tau_0]$ and $\mu\in\M_{q,+}^{r,m_0}$, all intermediate measures lie in a common $\M_{q,+}^{r',m'}$, where $\tau_0>0$, $r'<\infty$ and $m'>0$ depend only on $(r,m_0)$.
\item \emph{Near identity.} On each mass-restricted moment sublevel,
$d_{q,2}^*(T_\tau^{(j)}[\mu],\mu)\le C_j\tau^{\alpha_j}$
uniformly in $\mu$ for $\tau\in(0,\tau_0]$, for some $\alpha_j\in(0,1]$.
\item \emph{Continuity of the generator.} For each fixed $\varphi\in\mathcal D$, on each such sublevel,
$|G_j[\mu](\varphi)-G_j[\nu](\varphi)|\le L_{j,\varphi}d_{q,2}^*(\mu,\nu)$.
\end{enumerate}
Constants may depend on the sublevel. For balanced probability-valued operators the mass condition is automatic.
\end{assumption}

Conditions (A1)--(A3) control the first-order increments and intermediate measures; (A4) permits evaluating each generator at the original measure. They assert no existence, uniqueness, or convergence of finite-step iterates; the mass bound is needed because $W_2$ compares normalized laws. Section~\ref{subsec:canonical-examples} states the assumptions under which the three canonical classes satisfy them; the verifications are in Online Appendix~S3.

\subsection{Composition theorem}\label{subsec:composition-main}\label{app:composition-proof}

Successive operators act on different intermediate laws; the theorem justifies adding their separately computed generators, which enables the operator-wise Lyapunov bounds of Section~\ref{sec:lyapunov}.

\begin{theorem}[Generator of the composition]\label{thm:operator-composition}
Under Assumption~\ref{assumption:composite}, $G[\mu]=\sum_{j=1}^JG_j[\mu]$ is the composite pre-generator: for every $r<\infty$, $m_0>0$, and $\varphi\in\mathcal D$,
\[
\sup_{\mu\in\M_{q,+}^{r,m_0}}\left|\frac{\langle\varphi,T_\tau[\mu]\rangle - \langle\varphi,\mu\rangle}{\tau} - \sum_{j=1}^J G_j[\mu](\varphi)\right| \xrightarrow[\tau\downarrow0]{} 0.
\]
\end{theorem}

The proof telescopes the increments and uses (A3)--(A4) to replace generators at intermediate measures by generators at $\mu$.

\begin{proof}
Fix $r,m_0,\varphi$ and let $\mu^{(0)}=\mu$, $\mu^{(j)}=T_\tau^{(j)}[\mu^{(j-1)}]$. By (A2), all these measures, including $\mu$, lie in a common $\mathcal K=\M_{q,+}^{r',m'}$ for small $\tau$. Put $R_{j,\tau}(\eta):=\tau^{-1}\langle\varphi,T_\tau^{(j)}[\eta]-\eta\rangle-G_j[\eta](\varphi)$ and $\epsilon_j(\tau):=\sup_{\eta\in\mathcal K}|R_{j,\tau}(\eta)|$, so that $\epsilon_j(\tau)\to0$ by (A1). Telescoping,
\begin{equation}\label{eq:remainder}
\tau^{-1}\langle\varphi,T_\tau[\mu]-\mu\rangle-\sum_jG_j[\mu](\varphi)
=\sum_j\bigl(G_j[\mu^{(j-1)}](\varphi)-G_j[\mu](\varphi)\bigr)
+\sum_jR_{j,\tau}(\mu^{(j-1)}).
\end{equation}
The triangle inequality and (A3) give $d_{q,2}^*(\mu^{(j-1)},\mu)\le\sum_{k<j}C_k\tau^{\alpha_k}$ uniformly over the initial sublevel, so by (A4) the right side of~\eqref{eq:remainder} is at most $\sum_jL_{j,\varphi}\sum_{k<j}C_k\tau^{\alpha_k}+\sum_j\epsilon_j(\tau)\to0$ in absolute value. All constants are those of $\mathcal K$, hence independent of $\mu$; taking the supremum proves the theorem. The finite sum belongs to $\mathcal D'$.
\end{proof}

The theorem is a first-order consistency statement, not a convergence theorem for repeated finite steps.

\begin{remark}[Ordering and balancing]\label{rem:ordering-balancing-app}
Every ordering of the component maps for which (A1)--(A4), including intermediate moment and mass stability, hold has the same first-order pre-generator $\sum_jG_j$; the finite-$\tau$ maps need not commute and their splitting errors can differ. Balancing makes each canonical generator annihilate the constant test function~$1$ and hence preserves mass. If an unbalanced curve satisfies $\frac{d}{dt}\langle\varphi,\mu_t\rangle=G^{\rm un}[\mu_t](\varphi)$ with $m_t:=\mu_t(\X)>0$ and $\bar\mu_t:=\mu_t/m_t$, the quotient rule gives, wherever the derivatives exist,
\[
\frac{d}{dt}\langle\varphi,\bar\mu_t\rangle
=\frac{G^{\rm un}[\mu_t](\varphi)
-\langle\varphi,\bar\mu_t\rangle G^{\rm un}[\mu_t](1)}{m_t},
\]
so normalization adds a centered mass correction, and a balanced generator or Lyapunov rate cannot in general be copied to an unbalanced evolution.
\end{remark}

\subsection{General evolution equation and induced search-law dynamics}\label{subsec:general-evolution}\label{app:induced-intuition}

Theorem~\ref{thm:operator-composition} identifies the vector field associated
with the splitting. A continuous-time solution curve for this vector field is
therefore defined by the \emph{general evolution equation}
\begin{equation}\label{eq:general-evolution}
\frac{d}{dt}\langle\varphi,\mu_t\rangle = G[\mu_t](\varphi) = \sum_{j=1}^J G_j[\mu_t](\varphi)
\end{equation}
for all $\varphi\in\mathcal D$. Here a $\mathcal D$-weak solution means that $t\mapsto G[\mu_t](\varphi)$ is locally integrable and $\langle\varphi,\mu_t\rangle-\langle\varphi,\mu_s\rangle=\int_s^tG[\mu_u](\varphi)\,du$ for every such test and $s\le t$. No density is required, and identifying $G$ does not establish existence of a solution.

\begin{corollary}[Induced evolution on $\S$]\label{cor:induced-S}
Let $(\bar\mu_t)_{t\ge0}\subset\P_q(\X)$ be a probability-valued solution of~\eqref{eq:general-evolution}, let $\nu_t:=\Pi[\bar\mu_t]$ be the search law~\eqref{eq:search-law}, and let
$(K_*\phi)(x):=\int_\S\phi(y)\,K(x,dy)$ denote the kernel lift.
Then for every $\phi\in C_b(\S)$ such that $K_*\phi\in\mathcal{D}$, the map $t\mapsto\langle\phi,\nu_t\rangle$ is locally absolutely continuous and
\[
\frac{d}{dt}\langle\phi,\nu_t\rangle = G[\bar\mu_t](K_*\phi)\qquad\text{for a.e. }t\ge0.
\]
\end{corollary}

\begin{proof}
Fubini gives $\langle\phi,\nu_t\rangle=\langle K_*\phi,\bar\mu_t\rangle$. Since $K_*\phi\in\mathcal D$, the weak evolution equation yields $\langle\phi,\nu_t\rangle-\langle\phi,\nu_s\rangle=\int_s^tG[\bar\mu_\ell](K_*\phi)\,d\ell$, whose integrand is locally integrable by the weak-solution hypothesis; absolute continuity and the almost-everywhere derivative follow.
\end{proof}

For nonparametric methods the lift is the identity; for CMA-ES it propagates parameter updates through the Gaussian search kernel.

\subsection{Canonical operators: transport, reaction, and jumps}\label{subsec:canonical-examples}\label{subsec:canonical-conditions}

Here \emph{canonical} refers to three representative population-law updates: mutation moves individuals, selection reweights them, and recombination produces offspring from parents. We give their mass-preserving (\emph{balanced}) forms for $\mu\in\M_{q,+}(\X)$ with normalization $\bar\mu$ and step size $0<\tau\le1$, each with the coefficient assumption under which it satisfies Assumption~\ref{assumption:composite}. Figure~\ref{fig:landscape-operators}(b)--(d) shows one step of each; Table~\ref{tab:verified-operator-inventory} lists checked choices and their scope, none of which the framework requires. The verifications (Propositions~S3.3, S3.9, and~S3.18) are demonstrations of the calculus: their proofs, and the variants of selection and recombination that change the total mass, are in Online Appendix~S3 and are not needed for the general results of Sections~\ref{sec:operators}--\ref{sec:finite-budget}. We write $\M_{q,+}^r:=\{\mu\in\M_{q,+}(\X):\int(1+\|x\|^q)\,d\mu\le r\}$.

\begin{table}[t]
\centering\small
\setlength{\tabcolsep}{4pt}
\begin{tabularx}{\textwidth}{@{}>{\raggedright\arraybackslash}p{0.19\textwidth}>{\raggedright\arraybackslash}p{0.33\textwidth}X@{}}
\toprule
Component & Verified property & Requirements and scope\\\midrule
\multicolumn{3}{@{}l}{\textbf{Mutation:} noise can support exploration; objective-directed drift can support exploitation.}\\[3pt]
Gaussian (no drift) & Noise-only diffusion generator (Example~S3.4). & Required noise moments; regularity does not imply objective descent.\\[3pt]
Langevin & Potential-directed drift plus noise (Example~S3.5). & Lipschitz potential gradient, fixed positive-semidefinite preconditioner, and noise moments.\\[3pt]
CBO & Consensus attraction plus noise; target-bias criterion (Example~S3.6, Online Appendix~S3.6). & Regular consensus map and coefficients; convergence additionally needs the stated entry-class persistence.\\[3pt]\midrule
\multicolumn{3}{@{}l}{\textbf{Selection:} exploitation through reweighting preferred states; no new states are created.}\\[3pt]
Objective-value selection & Centered reweighting generator (Example~S3.10). & Bounded Lipschitz pressure; unbounded raw objectives require a separate argument.\\[3pt]
Rank, tournament, and soft truncation & Bounded pressure and regularized continuity (Examples~S3.11--S3.13). & Smoothed ranks require Lipschitz objective and utility; exact ranks need law and spatial bounds. Hard truncation is excluded.\\[3pt]\midrule
\multicolumn{3}{@{}l}{\textbf{Recombination:} exploration, exploitation, or both, depending on parents and the mixing rule.}\\[3pt]
Arithmetic and blend crossover (BLX) & Mixing Lipschitz and offspring moment bounds (Example~S3.19). & Compatible parent marginals and amplitude; extrapolation increases the mixing bound.\\[3pt]
Simulated binary crossover & Mixing and moment bounds (Example~S3.20). & Finite second and $q$-th spread-factor moments, plus parent-law and amplitude conditions.\\[3pt]
Differential-vector blending & Two- and three-parent affine bounds (Example~S3.21). & Mixing moments and parent-law conditions; not verification of full DE with greedy replacement.\\
\bottomrule
\end{tabularx}
\caption{Checked operators, grouped by mechanism. Exploration/exploitation roles describe possible effects, not guarantees of descent or discovery. Compatibility alone does not imply convergence.}\label{tab:verified-operator-inventory}
\end{table}

\paragraph{Mutation.}
A mutation operator moves each state by a Markov kernel $p_\tau(x,\cdot;\mu)$ on $\X$ that may depend on the population, $M_\tau[\mu](A):=\int_\X p_\tau(x,A;\mu)\,\mu(dx)$, and therefore preserves mass. In the canonical form, $p_\tau(x,\cdot;\mu)$ is the law of the perturbed state
\begin{equation}\label{eq:mutation-update-main}
x^+=x+\tau\,b_M(x;\bar\mu)+\sqrt{\tau}\,B_M(x;\bar\mu)\,\psi,
\end{equation}
where $b_M$ is a drift and $B_M$ scales a noise increment $\psi$ independent of $x$. The infinitesimal state-noise covariance is $a_M:=B_M\Sigma B_M^\top$ with $\Sigma:=\E[\psi\psi^\top]$. Setting $B_M=0$ gives deterministic motion, such as a gradient or consensus step.

\begin{assumption}[Diffusion-admissible mutation]\label{ass:mutation-main}
The noise increment $\psi$ takes values in $\R^{d'}$, is centered, has covariance $\Sigma\in\R^{d'\times d'}$, symmetric and positive semidefinite, and $\E\|\psi\|^{3}<\infty$. The coefficients $b_M\colon\X\times\P_3(\X)\to\R^{d_x}$ and $B_M\colon\X\times\P_3(\X)\to\R^{d_x\times d'}$ in~\eqref{eq:mutation-update-main} satisfy, on each $\M_{3,+}^r$ with constants depending on $r$ and with $\|\cdot\|_F$ the Frobenius norm:
\begin{enumerate}[label=\textup{(\roman*)},leftmargin=*,itemsep=2pt]
\item Linear growth. $\|b_M(x;\bar\mu)\|+\|B_M(x;\bar\mu)\|_F\le C(r)\,(1+\|x\|)$.
\item Joint Lipschitz continuity. There exists $L(r)<\infty$ such that $\|b_M(x;\bar\mu)-b_M(y;\bar\nu)\|+\|B_M(x;\bar\mu)-B_M(y;\bar\nu)\|_F \le L(r)\bigl(\|x-y\|+W_2(\bar\mu,\bar\nu)\bigr)$ for all $x,y\in\X$ and $\bar\mu,\bar\nu\in\P_3(\X)$.
\end{enumerate}
\end{assumption}

Under Assumption~\ref{ass:mutation-main}, Proposition~S3.3 verifies (A1)--(A4) for $M_\tau$ on third-moment sublevels, with near-identity exponent $\tfrac12$ and pre-generator
\[
G_M[\mu](\varphi)=\int_\X \nabla\varphi(x)\cdot b_M(x;\bar\mu)\,\mu(dx)+\frac12\int_\X\Tr\!\bigl(a_M(x;\bar\mu)\,D^2\varphi(x)\bigr)\,\mu(dx).
\]

\paragraph{Selection.}
A \emph{selection rate} $\Phi\colon\X\times\M_{q,+}(\X)\to\R$ favors lower-scoring states through normalized exponential reweighting,
\[
S_\tau^{\rm bal}[\mu](dx):=\mu(\X)\,\frac{e^{-\tau\Phi(x;\mu)}\mu(dx)}{\int_\X e^{-\tau\Phi(y;\mu)}\mu(dy)}.
\]
No state moves; its relative weight changes. Normalization yields the centered reaction term $-(\Phi-\bar\Phi)\mu$, where $\bar\Phi(\mu):=\mu(\X)^{-1}\int_\X\Phi(x;\mu)\,\mu(dx)$.

\begin{assumption}[Selection-rate assumptions]\label{ass:selection-main}
Fix $q\ge 2$. The selection rate $\Phi$ is measurable and satisfies:
\begin{enumerate}[label=\textup{(\roman*)},leftmargin=*,itemsep=2pt]
\item Bounded pressure. For every $r<\infty$ there exists $C_\Phi(r)<\infty$ such that
$|\Phi(x;\mu)|\le C_\Phi(r)$ for all $x\in\X$, $\mu\in\M_{q,+}^r$.
\item Selection-rate regularity. $\Phi(\cdot\,;\mu)$ is measurable in~$x$, and for each $r<\infty$ there exists $L_\Phi(r)<\infty$ such that $\sup_{x\in\X}|\Phi(x;\mu)-\Phi(x;\nu)|\le L_\Phi(r)\,W_2(\bar\mu,\bar\nu)$ for all $\mu,\nu\in\M_{q,+}^r$.
\item Spatial regularity of the selection rate. For every $r<\infty$ there exists $K_\Phi(r)<\infty$ such that for all $\nu\in\M_{q,+}^r$ and all $x,y\in\X$, $|\Phi(x;\nu)-\Phi(y;\nu)|\le K_\Phi(r)\,\|x-y\|\,\bigl(1+\|x\|^{q-1}+\|y\|^{q-1}\bigr)$.
\end{enumerate}
\end{assumption}

Under Assumption~\ref{ass:selection-main}(i), Proposition~S3.9 verifies (A1)--(A3) for $S_\tau^{\rm bal}$ for every $q\ge2$, with near-identity exponent $\tfrac12$ and pre-generator $G_S^{\rm bal}[\mu](\varphi)=-\int_\X(\Phi(x;\mu)-\bar\Phi(\mu))\varphi(x)\,\mu(dx)$; under (ii)--(iii) in addition, (A4) holds.

\paragraph{Recombination.}
Parents $(x,y)$ are drawn from a \emph{pair law} $\overline{\Gamma^{\mu}}\in\P_2(\X\times\X)$, the normalization of $\Lambda(x,y;\bar\mu)\,\bar\mu(dx)\,\bar\mu(dy)$ for a bounded, nonnegative, symmetric \emph{mating kernel} $\Lambda(\cdot,\cdot;\bar\mu)$, and produce the offspring $b(x,y;\xi)$ through a measurable mixing rule $b\colon\X\times\X\times\Xi\to\X$ with independent random input $\xi\sim\Prob_\xi$; for example, arithmetic crossover takes $b(x,y;\xi)=(1-\xi)x+\xi y$ with $\xi\in[0,1]$. With the offspring law scaled to the mass of $\mu$,
\[
R^{\rm bal}[\mu](A):=\mu(\X)\iint_{\X\times\X}\Prob_\xi\bigl(b(x,y;\xi)\in A\bigr)\,\overline{\Gamma^{\mu}}(dx,dy),
\qquad
R_\tau^{\rm bal}[\mu]:=(1-\tau)\mu+\tau R^{\rm bal}[\mu],
\]
the update replaces a fraction $\tau$ of the population by offspring, giving the jump term $R^{\rm bal}[\mu]-\mu$.

\begin{assumption}[Recombination admissibility]\label{ass:recombination-main}
Fix $q\ge 2$. The pair law and the mixing rule satisfy:
\begin{enumerate}[label=\textup{(\roman*)},leftmargin=*,itemsep=2pt]
\item \emph{Pair-law compatibility and Lipschitz continuity.} For every $\bar\mu\in\P_q(\X)$, $\int_\X\Lambda(x,y;\bar\mu)\,\bar\mu(dy)=1$ for $\bar\mu$-a.e.\ $x$, so that both marginals of $\overline{\Gamma^\mu}$ equal $\bar\mu$. Moreover, there exists $L_\Lambda<\infty$ such that for all $\bar\mu,\bar\mu'\in\P_q(\X)$,
\[
W_2\bigl(\overline{\Gamma^{\bar\mu}},\overline{\Gamma^{\bar\mu'}}\bigr)\le L_\Lambda\,W_2(\bar\mu,\bar\mu').
\]
\item \emph{Lipschitz mixing reproduction rule.} There exists $L_b<\infty$ such that $\E_\xi\bigl[\|b(x,y;\xi)-b(x',y';\xi)\|^2\bigr]\le L_b\bigl(\|x-x'\|^2+\|y-y'\|^2\bigr)$ for all $x,x',y,y'\in\X$.
\item \emph{$q$-moment growth of offspring.} Denote $w_q(x)=1+\|x\|^q$. There exists $C_b^{(q)}<\infty$ such that $\E_\xi\bigl[w_q(b(x,y;\xi))\bigr]\le C_b^{(q)}\bigl(w_q(x)+w_q(y)\bigr)$ for all $x,y\in\X$.
\end{enumerate}
\end{assumption}

Under Assumption~\ref{ass:recombination-main}, Proposition~S3.18 verifies (A1)--(A4) for $R_\tau^{\rm bal}$ for every $q\ge2$, with near-identity exponent $\tfrac12$ and pre-generator $G_R^{\rm bal}[\mu](\varphi)=\langle\varphi,R^{\rm bal}[\mu]\rangle-\langle\varphi,\mu\rangle$, exact in $\tau$.

\paragraph{The transport--reaction--jump equation.}
The splitting $T_\tau=M_\tau\circ R_\tau^{\rm bal}\circ S_\tau^{\rm bal}$ selects, recombines, then mutates, each component using the population distribution at that stage; ordering and balancing are discussed in Remark~\ref{rem:ordering-balancing-app}.

\begin{theorem}[Transport--reaction--jump decomposition]\label{thm:TRJ-main}
Let \(q\ge3\). Suppose Assumptions~\ref{ass:mutation-main}, \ref{ass:selection-main}, and \ref{ass:recombination-main} hold, and assume that the balanced splitting \(T_\tau=M_\tau\circ R_\tau^{\rm bal}\circ S_\tau^{\rm bal}\) satisfies Assumption~\ref{assumption:composite} on every sublevel \(\M_{q,+}^{r,m_0}\). Then:
\begin{enumerate}[label=\textup{(\roman*)},leftmargin=*,itemsep=2pt]
\item \emph{(Additive pre-generator.)} $T_\tau$ has pre-generator $G[\mu]=G_M[\mu]+G_S^{\rm bal}[\mu]+G_R^{\rm bal}[\mu]$ in $\mathcal D'$.
\item \emph{(Mean-field equation.)} The associated $\mathcal D$-weak evolution $\partial_t\mu_t=G[\mu_t]$ in $\mathcal D'$ is the \emph{balanced transport--reaction--jump equation}
\begin{equation}\label{eq:TRJ-main-example}
\partial_t\mu_t
=-\nabla\!\cdot(b_M\mu_t)
+\tfrac12\sum_{i,j=1}^{d_x}\partial_{ij}(a_{M,ij}\mu_t)
-(\Phi-\bar\Phi)\mu_t+R^{\rm bal}[\mu_t]-\mu_t
\end{equation}
in $\mathcal D'=(C_b^3(\X))'$, with the weak interpretation below, where $b_M$ and $a_M=B_M\Sigma B_M^\top$ are evaluated at $(\cdot\,;\bar\mu_t)$, $\Phi$ at $(\cdot\,;\mu_t)$, and $\bar\Phi$ at $\mu_t$.
\end{enumerate}
\end{theorem}

Transport--reaction--jump (TRJ) is shorthand for state motion (including diffusion), centered reweighting, and offspring replacement. For $q=3$, Propositions~S3.3, S3.9, and S3.18 verify the composition assumptions. For $q>3$ the mutation moment and metric estimates must be rechecked, since Proposition~S3.3 assumes only a third noise moment; finite $q$-th noise moments suffice. The weak-solution definition of Section~\ref{subsec:general-evolution} is used with $G=G_M+G_S^{\rm bal}+G_R^{\rm bal}$ given by the three formulas above, on curves with $\sup_{t\le T}\int(1+\|x\|^q)\,d\mu_t<\infty$ and $\inf_{t\le T}\mu_t(\X)>0$ on each finite horizon (automatic for probability laws). The derivatives in~\eqref{eq:TRJ-main-example} are defined by the component integral actions on $C_b^3(\X)$ and agree with distributional derivatives on compactly supported smooth tests; no density is required.

Theorem~\ref{thm:TRJ-main} is Theorem~S3.22 of the Online Appendix, where it is proved. It identifies an evolution equation, not a descent inequality; Section~\ref{sec:lyapunov} supplies the estimate that yields convergence. Under the same three assumptions and $q\ge3$, every nonnegative $C^3$ function with the polynomial derivative bounds of Theorem~\ref{thm:mean-field-convergence} is cutoff-admissible for $G_M$, $G_S^{\rm bal}$, and $G_R^{\rm bal}$ (Lemma~S3.23), so the Lyapunov principle applies to solutions of~\eqref{eq:TRJ-main-example}.

\section{Lyapunov Principle and Mean-Field Convergence}\label{sec:lyapunov}

A \emph{search-space gauge} $\Psi\colon\S\to[0,\infty]$ measures optimization error: the objective gap vanishes on $\S_*$, whereas $\Psi(y)=\|y-y_*\|^2$ targets one chosen minimizer. A \emph{state-space Lyapunov function} $\Upsilon\colon\X\to[0,\infty)$ tracks dissipation along the internal evolution. The sampling kernel connects the two via the \emph{lifted gauge} $\ell_\Psi(x):=\int_\S\Psi\,dK(x,\cdot)$. We write $\mathcal{V}_\Upsilon(\bar\mu):=\int\Upsilon\,d\bar\mu$ for the state-space Lyapunov functional and $\mathcal{E}_\Psi(\bar\mu):=\int\ell_\Psi\,d\bar\mu$ for the search error.
In the convergence theorem below we use two gauges in parallel: the objective-gap gauge $\Psi_f(y):=f(y)-f_*$ and, relative to a chosen minimizer $y_*\in\S_*$, the quadratic gauge $\Psi_*(y):=\|y-y_*\|^2$.

\subsection{Admissible unbounded tests}\label{subsec:cutoff-admissible}
The domain $\mathcal D=C_b^3(\X)$ need not contain an unbounded Lyapunov function; the following condition extends the generator identity $G=\sum_{j=1}^JG_j$ to one without assuming a chain rule or negative drift. Write $w_q(x):=1+\|x\|^q$.

\begin{definition}[Cutoff-admissible Lyapunov test]\label{def:cutoff-admissible}
Let $q\ge3$, let $\mathcal R\subseteq\X$, and let $\Upsilon\colon\X\to[0,\infty)$ be $C^3$ with $\|D^k\Upsilon(x)\|_k\le C_k(1+\|x\|^{q-k})$ for $k=0,1,2,3$, where $\|D^0\Upsilon\|_0=|\Upsilon|$ and the norms for $k=1,2,3$ are the Euclidean, Frobenius and multilinear operator norms. Fix a smooth ambient cutoff $\chi$ with $0\le\chi\le1$, $\chi=1$ on $B_1(0)$ and $\chi=0$ outside $B_2(0)$, restricted to $\X$, and put $\Upsilon_r(x):=\chi(x/r)\Upsilon(x)$ for $r\ge1$, so $\Upsilon_r\in\mathcal D$.

The function $\Upsilon$ is \emph{cutoff-admissible} for $G_1,\ldots,G_J$ on $\mathcal R$ if, for every finite $M$, the following hold on the class of $\mu\in\P_q(\X)$ with $\mu(\mathcal R)=1$ and $\int w_q\,d\mu\le M$:
\begin{enumerate}[label=\textup{(\roman*)},leftmargin=*,itemsep=2pt]
\item for each $j$, the limit $G_j[\mu](\Upsilon):=\lim_{r\to\infty}G_j[\mu](\Upsilon_r)$ exists and is finite;
\item for each $j$, $\sup_{r\ge1}|G_j[\mu](\Upsilon_r)|\le C_j(M)$ for a finite constant $C_j(M)$.
\end{enumerate}
Write $G[\mu](\Upsilon):=\sum_{j=1}^JG_j[\mu](\Upsilon)$ for the extended composite action.
\end{definition}

Along a $\mathcal D$-weak solution, $u\mapsto G[\bar\mu_u](\Upsilon_r)$ is measurable and locally integrable by the weak-solution hypothesis, so the extended action is measurable as a pointwise limit. The extension depends on the fixed $\chi$; its time integral along a solution does not (Lemma~\ref{lem:lyapunov-testing}), and for the canonical generators Lemma~S3.23 identifies it with the $\chi$-free integral formulas.

\subsection{Convergence theorem}\label{subsec:mf-theorem}
\begin{theorem}[Mean-field convergence]\label{thm:mean-field-convergence}
Let $q\ge3$ and let $(\bar\mu_t)\subset\P_q(\X)$ solve the general evolution equation~\eqref{eq:general-evolution}, with locally uniformly bounded $q$-th moments and support in a set $\mathcal R\subset\X$ for all $t\ge0$. Let $\Upsilon\ge0$ be $C^3$ with
$\|D^k\Upsilon(x)\|_k\le C_k(1+\|x\|^{q-k})$ for $k=0,1,2,3$, where $\|\cdot\|_k$ is the derivative-tensor norm. Suppose it is cutoff-admissible for the components $G_1,\dots,G_J$ of $G=\sum_jG_j$ on $\mathcal R$ (Definition~\ref{def:cutoff-admissible}) and, for some $\lambda>0$,
$G[\bar\mu](\Upsilon)\le-\lambda\mathcal V_\Upsilon(\bar\mu)$ whenever $\bar\mu(\mathcal R)=1$. Then
$\mathcal V_\Upsilon(\bar\mu_t)\le e^{-\lambda t}\mathcal V_\Upsilon(\bar\mu_0)$.
The following conclusions require only their respective compatibility inequalities on $\mathcal R$:
\begin{enumerate}[label=(\alph*),nosep,leftmargin=*]
\item \emph{Concentration (MF1).} If $\ell_{\Psi_*}\le C_{\mathrm{quad}}\Upsilon$ and $f$ satisfies Assumption~\ref{assumption:landscape}(i), choose any $R_0>0$, set $L_0=L_f(1+2\|y_*\|+2R_0)^s$ and $r_\varepsilon=\min\{R_0,\varepsilon/L_0\}$. Then
\[
1-p_\varepsilon(t)\le C_{\mathrm{quad}}r_\varepsilon^{-2}
e^{-\lambda t}\mathcal V_\Upsilon(\bar\mu_0).
\]
Here $p_\varepsilon(t)=\nu_t(\S_*^\varepsilon)$.\label{item:mf1}
\item \emph{Objective gap (MF2).} If $\ell_{\Psi_f}\le C_{\mathrm{obj}}\Upsilon$, then
$\int(f-f_*)\,d\nu_t\le C_{\mathrm{obj}}e^{-\lambda t}\mathcal V_\Upsilon(\bar\mu_0)$.
Markov's inequality also bounds $1-p_\varepsilon(t)$ by this quantity divided by $\varepsilon$, without basin assumptions.\label{item:mf2}
\item \emph{Quadratic error.} If $\ell_{\Psi_*}\le C_{\mathrm{quad}}\Upsilon$, then
$W_2^2(\nu_t,\delta_{y_*})\le C_{\mathrm{quad}}e^{-\lambda t}\mathcal V_\Upsilon(\bar\mu_0)$, giving MF3 when the minimizer is unique.\label{item:mf3}
\end{enumerate}
\end{theorem}

The decay follows from the Lyapunov identity of Lemma~\ref{lem:lyapunov-testing} and Gr\"onwall's lemma, and parts (a)--(c) from the kernel-lift identity of Lemma~\ref{lem:search-space-compatibility}; the proofs are in Appendices~\ref{app:lyap-decay}--\ref{app:lyap-modes}.

\subsection{Operator-wise closure}\label{app:lyap-closure}
The drift inequality of Theorem~\ref{thm:mean-field-convergence} is a verification target, not a presumed property of an optimizer. The following proposition assembles it from component-wise bounds whose bias terms are controlled by the Lyapunov functional.

\begin{proposition}[Operator-wise closure]\label{prop:operatorwise}
Suppose $G=\sum_{j=1}^JG_j$ on $\mathcal D$ and $\Upsilon$ is cutoff-admissible for these components on $\mathcal R$. For each $j$, suppose uniformly over the relevant laws supported in $\mathcal R$ that
\[
G_j[\bar\mu](\Upsilon)\le-\lambda_j\mathcal V_\Upsilon(\bar\mu)+c_j\mathfrak b_j(\bar\mu),
\qquad
0\le\mathfrak b_j(\bar\mu)\le\kappa_j\mathcal V_\Upsilon(\bar\mu),
\]
where $\lambda_j\in\R$ and $c_j,\kappa_j\ge0$. If $\lambda:=\sum_{j=1}^J(\lambda_j-c_j\kappa_j)>0$, then the dissipation bound in Theorem~\ref{thm:mean-field-convergence} holds with rate $\lambda$.
\end{proposition}

\begin{proof}
The bounded-test identity applied to $\Upsilon_r$ and the finite component limits give $G[\bar\mu](\Upsilon)=\sum_jG_j[\bar\mu](\Upsilon)$. Summing the assumed inequalities and using $c_j\ge0$,
$G[\bar\mu](\Upsilon)\le-\sum_j\lambda_j\mathcal V_\Upsilon+\sum_jc_j\mathfrak b_j\le-\sum_j(\lambda_j-c_j\kappa_j)\mathcal V_\Upsilon=-\lambda\mathcal V_\Upsilon$.
\end{proof}

Cutoff admissibility justifies unbounded testing, not negative drift; Lemma~S3.23 verifies it for the canonical examples. Compatibility then transfers state estimates to candidates; Section~\ref{sec:verified} makes these checks concrete. A local drift bound does not prove entry into its region; Part~(b) of Theorem~\ref{thm:mean-field-convergence} requires no basin or convexity assumption. When the drift inequality holds only up to a constant residual on a class of laws, the decay holds up to a floor (Remark~\ref{rem:residual-decay}); the CBO certificate of Section~\ref{sec:verified} is used in this form.

\section{From Mean-Field Certificates to Finite Evaluation Budgets}\label{sec:finite-budget}
Corollary~\ref{thm:complexity-main} turns mean-field concentration into a sufficient evaluation budget for finding an $\varepsilon$-optimal candidate, given an algorithm-specific state-law approximation bound. At $t_k$, write $\bar\mu_{t_k}^N$ for the empirical law of the $N$ particle states, $\nu_{t_k}^N:=\Pi[\bar\mu_{t_k}^N]$ for its search law, and let the algorithm draw a fresh conditionally i.i.d.\ batch of $N_k$ candidates from it ($N_k=N$ in the full-population regime), with $\tau_\varepsilon^{(N)}$ the first time a candidate lands in $\S_*^\varepsilon$. $\Prob^N$ and $\E^N$ refer to the finite-system probability space of Appendix~\ref{app:finite-population}.

\paragraph{Why the corollary applies.}
Theorem~\ref{thm:mean-field-convergence}(a) verifies the required mean-field concentration, with $K_\varepsilon=C_{\mathrm{quad}}r_\varepsilon^{-2}\V_\Upsilon(\bar\mu_0)$ and $k_\varepsilon:=\min\{k:K_\varepsilon e^{-\lambda t_k}\le\tfrac12\}$: from this burn-in index on, a draw from the limiting search law succeeds with probability at least $1/2$. Corollary~\ref{thm:complexity-main} transfers this to the finite system through kernel stability and control of mass near the target-set boundary (sufficient conditions in Proposition~S5.4); these verify transfer conditions, not the state-law approximation bound.

Positive initial probability near a minimizer, used in CBO analyses, is not an additional hypothesis here: quantitative contraction yields substantial target mass after burn-in, and the corollary controls its finite-system transfer. The deterministic burn-in threshold is not an assumed finite-population basin-hitting event (Appendix~\ref{app:fp-complexity}).

\begin{corollary}[Discounted evaluation-complexity criterion]
\label{thm:complexity-main}
Let $\delta\in(0,1)$, fix $m\ge k_\varepsilon$, put $T=t_m$, $A_\varepsilon:=\S_*^\varepsilon$, $c_\varepsilon:=f_*+\varepsilon$, and assume the concentration bound of Theorem~\ref{thm:mean-field-convergence}\ref{item:mf1} with $K_\varepsilon$ and $k_\varepsilon$ as above. Suppose:
\begin{enumerate}[label=(\roman*),nosep,leftmargin=*]
\item at each $t_k$, conditionally on the pre-batch history, the $N_k$ candidates are i.i.d.\ with law $\nu_{t_k}^N$;
\item $K(x,\cdot)\in\P_2(\S)$ and $W_2(K(x,\cdot),K(x',\cdot))\le L_K\|x-x'\|$ for all $x,x'\in\X$, with $L_K<\infty$;
\item $\S$ is a length space (e.g.\ $\R^d$ or convex); for some $r_0,\eta_{\rm lev},L_{T,\varepsilon},M_{T,\varepsilon}>0$ with $L_{T,\varepsilon}r_0\le\eta_{\rm lev}$, $f$ is $L_{T,\varepsilon}$-Lipschitz on the $r_0$-neighborhood of $\partial A_\varepsilon$, and for $k_\varepsilon\le k\le m$ the objective-value law $f_\#\nu_{t_k}$ has a density bounded by $M_{T,\varepsilon}$ on $[c_\varepsilon-\eta_{\rm lev},c_\varepsilon+\eta_{\rm lev}]$.
\end{enumerate}
Write $r_*(T,\varepsilon):=r_0\wedge(16L_{T,\varepsilon}M_{T,\varepsilon})^{-1}$, $\mathcal A_N(T):=\sup_{k_\varepsilon\le k\le m}\E^N[W_2^2(\bar\mu^N_{t_k},\bar\mu_{t_k})]$, and
\begin{equation}\label{eq:fp-modulus-hypothesis}
\mathfrak A_N^{\rm disc}(T,\varepsilon):=
\Bigl(\tfrac{8L_K}{r_*(T,\varepsilon)}\Bigr)^2\mathcal A_N(T)
\sum_{j=k_\varepsilon}^m e^{-\frac14\sum_{k=j+1}^mN_k}.
\end{equation}
If
\begin{equation}\label{eq:fp-criterion}
\mathfrak A_N^{\rm disc}(T,\varepsilon)\le\tfrac{\delta}{2}
\quad\text{and}\quad
\textstyle\sum_{k=k_\varepsilon}^mN_k\ge4\log\tfrac{2}{\delta},
\end{equation}
then $\Prob^N(\tau^{(N)}_\varepsilon\le t_m)\ge1-\delta$. With $N_k\equiv N$ the sum in~\eqref{eq:fp-modulus-hypothesis} is at most $(1-e^{-N/4})^{-1}$, and the post-burn-in budget is $N(m-k_\varepsilon+1)\ge4\log(2/\delta)$; total cost also includes burn-in. The proof is in Appendix~\ref{app:finite-population}.
\end{corollary}

\paragraph{Using the criterion.}
Supply $\mathcal A_N(T)$ and verify~\eqref{eq:fp-criterion}; the prefactor converts state-law error to target-mass error, and late errors weigh more than early ones. If the best-so-far value includes these batches, $\Prob^N(B_m>\varepsilon)\le\delta$. No operator and no run uses $f_*$ or $y_*$: the criterion, the measured bound of Section~\ref{sec:numerics}, and the best-so-far bound apply to any sublevel set $\{f\le c\}$ in place of $\S_*^\varepsilon$, with the burn-in condition taken from runs, so a known lower bound $f_{\rm lb}\le f_*$ and $c=f_{\rm lb}+\varepsilon$ certify the true target. Online Appendix~S2 is a self-contained guide to checking~\eqref{eq:fp-modulus-hypothesis}--\eqref{eq:fp-criterion}: how to obtain each ingredient with or without derivatives, what propagation-of-chaos theorems supply for $\mathcal A_N(T)$, how to estimate it from a reference run, and a worked example. The worked example shows the transfer constants to be far from tight at practical population sizes, so the corollary is structural and certification in practice uses the measured route of Section~\ref{sec:numerics}.

\paragraph{Structure of the proof.}
The proof (Appendix~\ref{app:finite-population}) has two steps. The first is a discovery recursion that uses no approximation theory (Proposition~\ref{prop:fp-discounted}): with the \emph{inadequate-mass probability} $a_{N,k}(\varepsilon):=\Prob^N\bigl(\nu^N_{t_k}(\S_*^\varepsilon)<\tfrac12p_\varepsilon(t_k)\bigr)$, the probability that the finite search law assigns less than half the mean-field success mass to the target set, and the \emph{survival factor} $\rho_k(\varepsilon):=\bigl(1-\tfrac12p_\varepsilon(t_k)\bigr)^{N_k}$, which bounds the conditional probability that all $N_k$ fresh candidates miss when the mass is adequate,
\[
\Prob^N\bigl(\tau_\varepsilon^{(N)}>t_m\bigr)\le\sum_{j=k_\varepsilon}^m a_{N,j}(\varepsilon)\prod_{k=j+1}^m\rho_k(\varepsilon)+\prod_{k=k_\varepsilon}^m\rho_k(\varepsilon),
\qquad m\ge k_\varepsilon .
\]
An inadequate-mass event at batch $j$ is discounted by the survival factors of the later batches, and the final product is the pure-sampling term. The second step bounds $a_{N,j}(\varepsilon)$ by $(8L_K/r_*)^2\mathcal A_N(T)$ through hypotheses (ii)--(iii) (Corollary~\ref{cor:fp-cutoff}); this is where the state-law approximation enters, and it is the conservative step. Section~\ref{sec:numerics} therefore measures $a_{N,j}(\varepsilon)$ directly and keeps the recursion.

\section{Numerical Support: Constructed Assemblies}\label{sec:numerics}\label{subsec:numerics-main}
The canonical checks in Section~\ref{subsec:canonical-examples} are reusable; the remaining drift and region checks depend on the target (mean gap, quadratic gauge, or discovery). This section constructs assemblies of checked operators with individually switched components, so that every effect is studied component by component, and tests three predictions of the theory on them: that the generators add (Theorem~\ref{thm:operator-composition}); that the drift coefficient of Theorem~\ref{thm:mean-field-convergence} is a property of the assembly, the problem, and the population geometry; and that discovery by a finite population follows the recursion behind Corollary~\ref{thm:complexity-main}. The complete study, with protocols, every problem, and further diagnostics, is in Online Appendix~S1.

\paragraph{Assemble once, reuse the checks.}
We combine four components from the verified classes (Table~\ref{tab:assembly-operators}): derivative-free drift D along a cutoff-normalized central-difference field, Gaussian diffusion H, bounded-score selection S, and midpoint recombination R. D is always active; a \emph{switch} is one of the three parameters of Table~\ref{tab:assembly-operators} that turns H, S, or R on or off (zero makes the component the identity), and the switches give eight assemblies, applied in the orders $M_\tau R_\tau S_\tau$ and $M_\tau S_\tau R_\tau$ with $M_\tau=H_\tau D_\tau$; Lemma~S1.1 reuses the canonical proofs. Table~\ref{tab:assembly-operators} gives the TRJ equation of the full assembly and its component generators $G_D,G_H,G_S,G_R$; each switch deletes one term.

\begin{table}[t]
\centering\footnotesize
\setlength{\tabcolsep}{3pt}
\begin{tabular}{@{}cll@{}}\toprule
 & Finite-step update of $\nu$ & $G_j[\nu](f)$\\\midrule
D & $\operatorname{Law}(X+\tau b(X))$ & $\int\nabla f\cdot b\,d\nu$\\
H & $\operatorname{Law}(X+\sigma\sqrt\tau Z)$ & $\tfrac12\sigma^2\int\Delta f\,d\nu$\\
S & $e^{-\omega\tau\Phi}\nu\big/\!\int e^{-\omega\tau\Phi}d\nu$ & $-\omega\operatorname{Cov}_\nu(f,\Phi)$\\
R & $(1-\gamma\tau)\nu+\gamma\tau\operatorname{Law}(\tfrac{X+X'}2)$ & $\gamma[\E f(\tfrac{X+X'}2)-\int f\,d\nu]$\\
\bottomrule
\end{tabular}
\caption{Operator dictionary of the assemblies: $\nu$ is the law entering a component, $X,X'\sim\nu$ independent, $Z\sim\mathcal N(0,I_d)$, $b=-\chi_{\rm cut}\widehat g/\sqrt{1+\|\widehat g\|^2}$ with a smooth cutoff $\chi_{\rm cut}$ and the central-difference field $(\widehat g)_j(z)=[f(z+r_{\rm fd}\mathbf e_j)-f(z-r_{\rm fd}\mathbf e_j)]/(2r_{\rm fd})$ (Online Appendix~S4.5), and $\Phi=f/(1+f)$; $\sigma\in\{0,0.2\}$, $\omega,\gamma\in\{0,1\}$, and zero makes H, S, or R the identity. The full assembly's TRJ equation~\eqref{eq:TRJ-main-example} is $\partial_t\bar\mu_t=-\nabla\!\cdot(b\bar\mu_t)+\tfrac{\sigma^2}2\Delta\bar\mu_t-\omega(\Phi-\bar\Phi)\bar\mu_t+\gamma(R^{\rm bal}[\bar\mu_t]-\bar\mu_t)$ with $R^{\rm bal}[\bar\mu]=\operatorname{Law}(\tfrac{X+X'}2)$; its terms are generated by $G_D,G_H,G_S,G_R$ (rows), whose actions on $f$ (last column) add to $G[\mu](f)$ by Theorem~\ref{thm:operator-composition}.}\label{tab:assembly-operators}
\end{table}

\paragraph{Problems and protocol.}
The panel has ten problem--dimension instances with $f_*=0$ (Table~S1): separable and rotated nonconvex Polyak--\L{}ojasiewicz (PL) functions and a multimodal quartic in dimensions four and eight, Rastrigin, a rotated Rosenbrock function, and unequal wells in dimension four, and a periodic function in dimension one. Each is started from a broad Gaussian population and from a clustered two-group mixture of $N=32$ equally weighted individuals, with 32 calibration and 128 validation populations per instance and family. For every population, assembly, order $o$, and step $\tau=0.1/2^\ell$, $\ell=0,\ldots,7$, the change of the mean gap $V(\mu)=\int f\,d\mu$ per unit step, $D_{\tau,o}=[V(T_{\tau,o}\mu)-V(\mu)]/\tau$, is computed in closed form and compared with the generator prediction $G[\mu](f)=\sum_jG_j[\mu](f)$. Every comparison is \emph{paired}: the quantities compared are computed from the same initial population and, where candidates are sampled, from the same random numbers, and their difference is taken within that population before anything is averaged. The trajectory study runs every assembly as an evolving population of size $N\in\{16,32,64,128\}$ for 200 generations at $\tau=0.1$; a \emph{cell} is one combination of instance, initialization family, assembly, order, and $N$, run in 128 independent repetitions (1,280 cells). Figure~\ref{fig:main-results} shows the results on two four-dimensional problems: the multimodal quartic $\sum_j(x_j^2-1)^2$, whose $2^d$ minimizers make mixing of good points harmful, and the unequal wells, which has a competing worse well.

\begin{figure}[t]
\centering
\includegraphics[width=\textwidth]{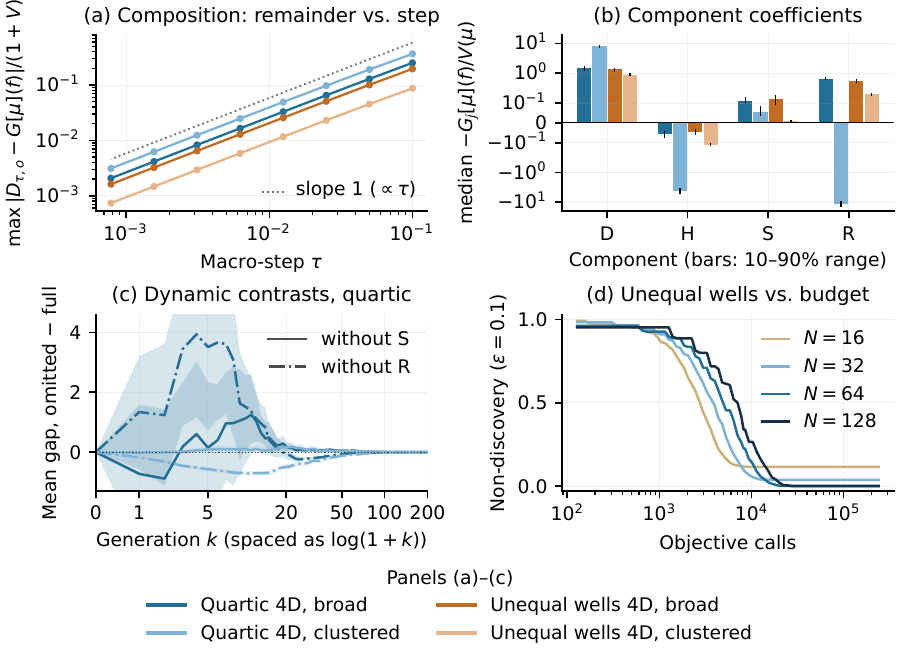}
\caption{Constructed assemblies on two panel problems from broad (dark) and clustered (light) initial populations. (a) Worst normalized remainder $|D_{\tau,o}-G[\mu](f)|/(1+V)$ over 128 validation populations, eight assemblies, and both orders, where $V(\mu)=\int f\,d\mu$, $D_{\tau,o}=[V(T_{\tau,o}\mu)-V(\mu)]/\tau$ is the exact change per unit step in order $o$, and $G[\mu](f)$ its first-order prediction (Theorem~\ref{thm:operator-composition}); the dotted line has slope one. (b) Median component coefficients $-G_j[\mu](f)/V(\mu)$ (symmetric log; positive means mean-gap reduction). (c) Quartic: run-averaged mean gap with S or R omitted minus the full assembly, differenced run by run from the same initial population ($N=32$, 128 starts, 95\% bootstrap bands). (d) Unequal wells, broad start: fraction of runs whose archive of evaluated points has no $\varepsilon$-good point after the given number of objective calls, by population size $N$ (averaged over assemblies, orders, and 128 runs; legend below: (a)--(c)). }
\label{fig:main-results}
\end{figure}

\paragraph{The additive prediction survives refinement and reordering.}
Theorem~\ref{thm:operator-composition} predicts that the exact change $D_{\tau,o}$ of $V(\mu)=\int f\,d\mu$ per unit step tends to $G[\mu](f)=\sum_jG_j[\mu](f)$ for both orders. In Figure~\ref{fig:main-results}(a) the worst remainder over 128 validation populations and all assemblies decreases linearly in the step, as does the difference between the orders (Figure~S1 shows every problem).

\paragraph{The drift coefficient depends on the assembly, the problem, and the population geometry.}
The ratios $-G_j[\mu](f)/V(\mu)$ in Figure~\ref{fig:main-results}(b) measure first-order mean-gap reduction and attribute it to the components; their sum over the components is the quantity that the drift hypothesis of Theorem~\ref{thm:mean-field-convergence} with $\Upsilon=f$ bounds below by $\lambda$. Drift and selection reduce the gap and diffusion increases it in every case, whereas recombination reverses sign: helpful on broad quartic clouds and on the unequal wells, strongly harmful on two-cluster quartic populations, where midpoints of good points in different wells are poor. For the recombination-enabled assemblies on two-cluster quartic populations the summed coefficient is negative, so no drift coefficient $\lambda>0$ exists for those assemblies on that family; the calculus predicts this sign and the finite-step check reproduces it. When the constants of Theorems~\ref{thm:operator-composition} and~\ref{thm:mean-field-convergence} are not available analytically, as here, they can only be estimated from runs. With $\lambda$ estimated by the minimum $\lambda_{\rm cal}$ of the summed coefficient over the 64 calibration populations of an instance and assembly, and the first-order error constant by the maximal normalized remainder, about $1.4\%$ (drift) and $2.2$--$2.4\%$ (remainder) of the validation populations fall outside the frozen estimates; for the drift this is the rate that exchangeability of calibration and validation populations predicts, since a new population falls below the minimum of 64 with probability $1/65$. This is evidence that the calibrated constants describe the family, not a certificate, which needs the infimum over the whole class along the trajectory (Online Appendix~S1.4 gives the distributions of the estimated coefficients for every problem and assembly, and Online Appendix~S2.6 the estimation procedure).

\paragraph{The value of a component depends on the target.}
Best-so-far gains are a separate diagnostic: noise raises the mean gap in every case yet often improves a fresh batch's best value, and recombination on two-cluster quartic populations has a positive mean cost yet improves the batch's best value in more than a third of the paired batches (Online Appendix~S1.5). Along a run (Figure~\ref{fig:main-results}(c)), removing R from clustered quartic starts lowers the mean gap while raising the variance, i.e.\ stronger agreement with worse objective quality; on broad starts the R contrast changes sign, and S helps throughout (Online Appendix~S1.7).

\paragraph{Finite populations: the measured discovery bound.}
Corollary~\ref{thm:complexity-main} says that a finite population discovers an $\varepsilon$-good point once its search law keeps enough mass on $\varepsilon$-good candidates. The trajectory study measures that success mass directly and turns it into a confidence bound on the probability of non-discovery, without any mean-field approximation. Proposition~\ref{prop:fp-discounted} bounds $\Prob^N(\tau_\varepsilon^{(N)}>t_m)$, the probability that no $\varepsilon$-good candidate has been found by generation $m$ ($\tau^{(N)}_\varepsilon$ is the discovery time of Corollary~\ref{thm:complexity-main}), through the inadequate-mass probabilities $a_{N,k}(\varepsilon)$ and the survival factors $\rho_k(\varepsilon)$. The study measures both. With $p^N_k:=\nu^N_{t_k}(\S_*^\varepsilon)$ the success mass of the proposal law at generation $k$ and $N_k$ the size of the batch drawn from it, set
\[
a_{N,k}(\varepsilon):=\Prob^N\bigl(p^N_k<\tfrac14\bigr),\qquad \rho_k(\varepsilon):=(3/4)^{N_k},
\]
the proposition's quantities with the fixed adequate-mass threshold $\tfrac14$ in place of $\tfrac12p_\varepsilon(t_k)$ (the batch consists of independent draws from $\nu^N_{t_k}$ given the history, selection being sampled exactly; for the offspring $N_k=N$). Running the conditional recursion~\eqref{eq:discovery-recursion-app} from generation $1$ with these ingredients, \eqref{eq:fp-discounted-exact} becomes
\[
\Prob^N(\tau_\varepsilon^{(N)}>t_m)\le\sum_{j=1}^{m}a_{N,j}(\varepsilon)\prod_{k=j+1}^{m}\rho_k(\varepsilon)+\prod_{k=1}^{m}\rho_k(\varepsilon).
\]
The $a_{N,j}(\varepsilon)$ are estimated. In every generation $j$, 128 \emph{probe} draws from $\nu^N_{t_j}$ (discarded, so the trajectory is unaffected) give a $\operatorname{Binomial}(128,p^N_{ij})$ count in repetition $i$, and the inadequate-mass event $\{p^N_{ij}<\tfrac14\}$ is flagged whenever the count's exact lower confidence limit at error $\eta_{\rm diag}=10^{-3}$ lies below $\tfrac14$; the flag fires whenever the event holds except with probability $\eta_{\rm diag}$, so $a_{N,j}(\varepsilon)\le\E[\text{flag}]+\eta_{\rm diag}$, and the flag frequency over the 128 repetitions has an exact binomial upper limit $\bar a_{N,j}(\varepsilon)$ for $\E[\text{flag}]$. For constant batches $N_k\equiv N_\bullet$ write $\rho:=(3/4)^{N_\bullet}$ and let $m_\rho$ be the number of final generations whose discount $\rho^{m-j}$ exceeds $10^{-3}$; a union bound over these generations (earlier $a_{N,j}(\varepsilon)$ are bounded by one) gives, with confidence $95\%$,
\begin{equation}\label{eq:empirical-discovery-confidence-v11}
\Prob^N(\tau_\varepsilon^{(N)}>t_m)\le\min\Bigl\{1,\ \rho^m+\sum_{\ell=0}^{m_\rho-1}\rho^\ell\min\{1,\bar a_{N,m-\ell}(\varepsilon)+\eta_{\rm diag}\}+\sum_{\ell=m_\rho}^{m-1}\rho^\ell\Bigr\}:
\end{equation}
the residual term $\rho^m$, the chance that every batch missed although the mass was adequate, plus the discounted upper limits of the inadequate-mass frequencies. The bound is computed per cell, tolerance, and endpoint (generations $25,50,100,200$ prespecified); a simultaneous version at generation 200 divides the error by the number of cells and batch streams ($1{,}280\times5$).

\paragraph{Reliability against cost, and what the bound certifies.}
Running every assembly as an evolving population shows a reliability--cost trade-off (Figure~\ref{fig:main-results}(d)): on the unequal wells, small populations discover an $\varepsilon$-good point sooner per evaluation but some runs never do, whereas larger populations remove these failures at higher cost. Non-discovery by generation 200 occurs only on unequal wells and Rastrigin. At generation 200 the measured bound~\eqref{eq:empirical-discovery-confidence-v11} certifies a non-discovery probability of at most $0.05$ (pointwise $95\%$) in 846 of the 1,280 cells. It is sufficient, not necessary: among the cells with $N=32$, a sixth to a fifth of those with no observed failure have a bound above $0.5$, because their measured success mass stayed below the adequate-mass threshold $\tfrac14$ or was not resolved by 128 probes. Discounting by the survival factors is what makes the bound informative: 797 of the 846 certified cells have an undiscounted sum of inadequate-mass frequencies of at least one and would not be certified without it (Online Appendix~S1.6: budget curves, failures by cell, and the dependence on the batch size; Online Appendix~S2.5: step-by-step computation).

\section{Verified Instances}\label{sec:verified}\label{subsec:instances}
Each instance supplies what Theorem~\ref{thm:mean-field-convergence} and Corollary~\ref{thm:complexity-main} leave open: a verified drift inequality, lift constants, and a finite-$N$ approximation bound or recursion. Statements and proofs are in Online Appendix~S4. Each instance is also run as specified in its statement, without box projection, annealed temperatures, or added exploration noise, against the proved envelope where explicit constants are available; no rate or residual is fitted (Figure~\ref{fig:verified-trajectories}; run details in Online Appendix~S4.8).

\begin{figure}[t]
\centering
\includegraphics[width=\textwidth]{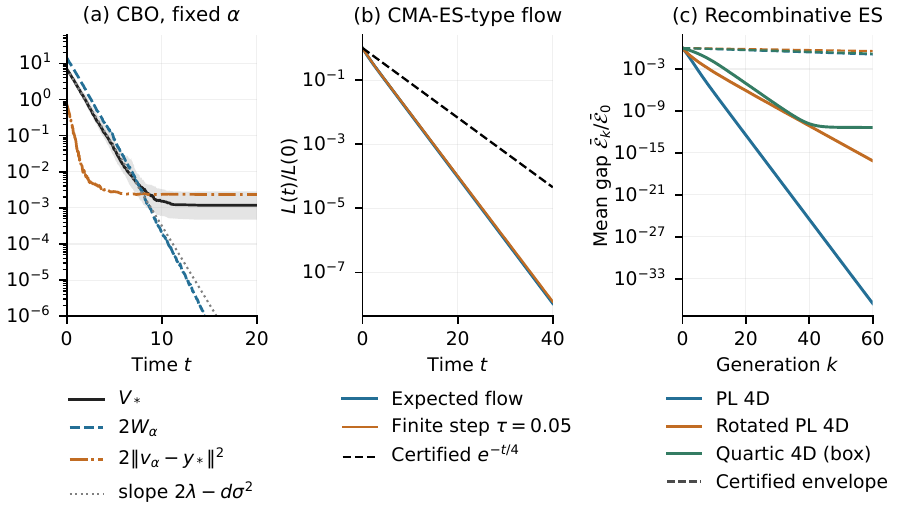}
\caption{Algorithm-level trajectories of the verified instances. (a) Isotropic CBO at fixed inverse temperature $\alpha=50$ on separable PL, $d=4$: medians over 64 runs of $\V_*$ (band: 10--90\% quantiles), $2\W_\alpha$ and $2\mathfrak b_\alpha$; the dotted reference line has slope $2\lambda-d\sigma^2$. (b) The CMA-ES-type flow of Example~S4.5: the expected flow from five initial states in the invariant region, the median and 10--90\% band of 256 stochastic runs with $\tau=0.05$, and the certified envelope $e^{-t/4}$; observed rate of the expected flow $0.457$--$0.459$ against the certified $1/4$. (c) The recombinative ES: run averages of the mean objective gap over 256 runs (solid) and the certified envelope~(S43) averaged over the same runs (dashed), both divided by the average initial gap; certified factors $\rho_{\rm RES}$ $0.965$ (PL), $0.984$ (rotated PL), $0.970$ (quartic) against observed average contraction $0.24$, $0.50$, $0.56$ over the first 20 generations.}\label{fig:verified-trajectories}
\end{figure}

\subsection{CMA-ES-type flows}\label{subsec:cma-main}
The state $x=(m,C,\sigma,p_\sigma,p_c)$ holds mean, covariance shape, scale, and two step histories; candidates follow $K(x,\cdot)=\mathcal N(m,\sigma^2C)$, and $G^{\rm CMA}$ is the expected generator of the scaled rank-$\mu$ update (Proposition~S4.1). For $\Upsilon_m(x)=\|m-y_*\|^2$ and $\Upsilon_c(x)=\sigma^2\Tr C$ the Gaussian lift is exactly $\Upsilon_m+\Upsilon_c$, so the covariance channel must be controlled too. Proposition~S4.4 is a two-channel certificate: if, on a support-forward-invariant region, $G^{\rm CMA}[\delta_x](\Upsilon_m)\le-\lambda_m\Upsilon_m+\gamma_m\Upsilon_c$ and $G^{\rm CMA}[\delta_x](\Upsilon_c)\le\gamma_c\Upsilon_m-\lambda_c\Upsilon_c$ with $\lambda_m\lambda_c>\gamma_m\gamma_c$, then $\Upsilon=\Upsilon_m+\alpha\Upsilon_c$ with a suitable channel weight $\alpha$ satisfies the drift hypothesis of Theorem~\ref{thm:mean-field-convergence} with $\lambda=\lambda_{\rm CMA}>0$, and the exact lift gives $\ell_{\Psi_*}\le\max\{1,\alpha^{-1}\}\Upsilon$, so, under the proposition's integrability conditions, parts (a) and (c) yield $\varepsilon$-concentration and $W_2$ convergence of the search law. Online Appendix~S4.1 reduces the channel bounds to ranked-step inequalities that convexity alone does not give; Example~S4.5 verifies them with an invariant region and rate $1/4$ on $1-\cos y$, confirmed by Figure~\ref{fig:verified-trajectories}(b). These are expected-flow results; finite-step CMA-ES also needs the modulus $\mathcal A_N(T)$ of Corollary~\ref{thm:complexity-main} (Online Appendix~S2).

\paragraph{The certified rate holds with room.}
From the example's initial state and four corners of the invariant region, the expected parameter flow stays in the region and satisfies $L(t)\le e^{-t/4}L(0)$ at every output time, with $L=\Upsilon_m+\Upsilon_c=(m-y_*)^2+\sigma^2C$ the lifted quadratic error, at nearly twice the certified rate (Figure~\ref{fig:verified-trajectories}(b)). The stochastic finite-step update is not covered by the example: with $\tau=0.05$ only a minority of runs remain in the invariant region, yet their median decay matches the expected flow, so the invariant region is sufficient for the proof, not necessary for decay.

\subsection{Derivative-free recombinative ES}\label{subsec:res-main}
Here $\X=\S=\R^d$, $K(x,\cdot)=\delta_x$, $\Psi_f=f-f_*$, and $\mathcal E_k^N=N^{-1}\sum_i\Psi_f(X_i^k)$ for $N$ individuals. Parents are resampled independently with positive, nonincreasing objective-gap weights, blended as $Z_i^k=\Xi_i^kX_{I_i^k}^k+(1-\Xi_i^k)X_{J_i^k}^k$ with an independent $\Xi_i^k\in[0,1]$, and mutated by a surrogate-gradient step $X_i^{k+1}=Z_i^k-h\beta\,\widehat g_k(Z_i^k)$, where $\widehat g_k$ uses objective values only (e.g.\ $\widehat g$ of Table~\ref{tab:assembly-operators}). Proposition~S4.8 is a post-entry finite-generation certificate: if from a deterministic generation $k_{\rm ent}$ on all individuals, recombined points, and mutation segments stay in a region $U$ on which blending two points changes the average gap by at most the factor $\rho_{\rm rec}$ plus $c_{\rm rec}$ and a mutation step contracts the gap by $\rho_{\rm mut}<1$ plus $c_{\rm mut}$, with $\rho_{\rm RES}:=\rho_{\rm rec}\rho_{\rm mut}<1$, then $\E[\mathcal E_{k+1}^N\mid\mathcal F_k^N]\le\rho_{\rm RES}\mathcal E_k^N+c_{\rm mut}+\rho_{\rm mut}c_{\rm rec}$ given the evaluated-population history $\mathcal F_k^N$. This is the finite-population counterpart of Theorem~\ref{thm:mean-field-convergence}(b) with $\mathcal E_k^N$ in the role of $\mathcal V_\Upsilon$; it replaces the mean-field route of Corollary~\ref{thm:complexity-main}, needs no $\mathcal A_N(T)$, and lets recombination increase the gap if mutation compensates. Explicit constants follow from smoothness, a Polyak--\L{}ojasiewicz bound, and a surrogate-error bound on $U$ (Lemma~S4.10); Online Appendix~S4.7 verifies them for nonconvex PL families and the quartic in an invariant box; Corollary~S4.9 gives the evaluation count. Figure~\ref{fig:verified-trajectories}(c) shows these unfitted envelopes hold but are very conservative: the certified factors are close to one, the observed decay orders of magnitude faster.

\paragraph{Valid but conservative envelopes.}
For the three certified instances, separable and rotated PL in $\R^4$ and the quartic in the invariant box $U=[0.4,1.2]^4$, the run-average gap stays below the certified envelope at every generation, and every individual, recombined point, and mutation segment of the quartic runs stays in $U$, as proved. The envelopes are loose because the constants behind $\rho_{\rm RES}$ are global or box-wide worst cases and the recombination bound holds for every parent pair.

\subsection{CBO}\label{subsec:cbo-main}
CBO is a pure-mutation instance drawn to the consensus $v_\alpha(\bar\mu)=\int xe^{-\alpha f}d\bar\mu/\int e^{-\alpha f}d\bar\mu$ at inverse temperature $\alpha$, drift $-\lambda h^{\rm CBO}_{\bar\mu}(x)(x-v_\alpha)$, $h^{\rm CBO}_{\bar\mu}\in[\underline h,1]$, and noise $\sigma\|x-v_\alpha\|$. With $\V_*$, $\W_\alpha$ the mean squared distances to $y_*$, $v_\alpha$ and bias $\mathfrak b_\alpha=\|v_\alpha-y_*\|^2$, the generator identity gives $G_M^{\rm CBO}[\bar\mu](\Upsilon_*)\le-a_\zeta\W_\alpha+(\lambda^2/\zeta)\mathfrak b_\alpha$ for $\Upsilon_*(x)=\|x-y_*\|^2$, $0<\zeta<2\lambda\underline h-d\sigma^2$, and $a_\zeta:=2\lambda\underline h-d\sigma^2-\zeta$. Under Assumption~\ref{assumption:basin}, an entry condition and a low-temperature condition on $\alpha$, Proposition~S4.13 closes the bias as $\mathfrak b_\alpha\le\kappa_{\alpha,r}\W_\alpha+c_{\alpha,r}$ (larger $\alpha$ shrinks $\kappa_{\alpha,r}$; $c_{\alpha,r}\ge2r^2$). If $a_\zeta>(\lambda^2/\zeta)\kappa_{\alpha,r}$, then $G_M^{\rm CBO}[\bar\mu](\Upsilon_*)\le-\lambda_{\rm CBO}\V_*+O(c_{\alpha,r})$ with explicit $\lambda_{\rm CBO}>0$: Theorem~\ref{thm:mean-field-convergence}'s drift hypothesis up to a residual, and Remark~\ref{rem:residual-decay} gives concentration and $W_2$ convergence to a floor of order $c_{\alpha,r}$ while entry persists.

\paragraph{Attraction to a biased weighted target.}
For isotropic CBO at fixed $\alpha$ on separable PL in $\R^4$, the identity $\V_*\le2\W_\alpha+2\mathfrak b_\alpha$ holds at every step of every run: median $\V_*$ first decreases at close to the centered coefficient $2\lambda-d\sigma^2$ while the population contracts toward the weighted consensus, then levels off at the persistent consensus bias while $\W_\alpha$ keeps contracting (Figure~\ref{fig:verified-trajectories}(a)), the fixed-temperature residual of Proposition~S4.13. No certified envelope is drawn: under global quadratic growth every population of the entry class has $\V_*$ below the residual level, so the certificate controls persistence of the class, not a visible decay of $\V_*$.

\section{Discussion and Limitations}\label{sec:conclusion}
\paragraph{What the calculus provides.}
The composition theorem turns the analysis of a population-based optimizer into checks on its components: each operator is verified once against (A1)--(A4), its generator is computed once, and a new combination inherits both. The Lyapunov principle then asks for one inequality per component and one compatibility inequality for the sampling kernel. The same decomposition attributes performance: the component actions $G_j[\mu](\Upsilon)$ say, at any population, which operator reduces the chosen error and which increases it, where an evaluation of the whole algorithm cannot. The attribution is first order and local to a population; along a run, removing an operator also changes the populations on which the others act. The benefit is reuse, and the cost is sharpness: a component-wise bound holds for every law in a class and every parent pair, and the certified factors of the recombinative ES are close to one while the observed decay is much faster (Figure~\ref{fig:verified-trajectories}(c)).

\paragraph{What is not proved.}
Theorem~\ref{thm:operator-composition} is a first-order consistency statement. It identifies the evolution equation of a splitting but establishes neither existence or uniqueness of its solutions nor convergence of repeated finite steps, and the one-step study confirms first-order consistency for the tested operators and laws, not a general second-order remainder estimate. Theorem~\ref{thm:mean-field-convergence} takes a solution, its support region, and the drift inequality as hypotheses; entry into the region and persistence in it are separate checks, which the instances of Section~\ref{sec:verified} carry out only in part: the CMA-ES-type certificate with an explicit invariant region is a one-dimensional expected-flow result and does not cover the stochastic finite-step update, and the CBO certificate holds up to a residual floor while its entry class persists. Hard ranks, truncation, greedy replacement, and memory rules are outside the verified operator classes unless smoothed or treated by a separate generator argument.

\paragraph{The finite-budget criterion is structural.}
Corollary~\ref{thm:complexity-main} isolates what a finite-population guarantee needs, a state-law approximation estimate and a transfer from state-law error to target-set mass, but the transfer constant is far from tight at practical population sizes (Online Appendix~S2.3), and the available propagation-of-chaos estimates for selection and crossover are stated in test-function or bounded-Lipschitz distances that must be converted to $W_2$ (Online Appendix~S2.1). Certification at practical sizes therefore uses the measured discovery bound of Section~\ref{sec:numerics}, which is an empirical guarantee with a stated confidence level, or a direct finite-population recursion such as that of the recombinative ES. Sharpening the $W_2$-to-mass transfer is the open problem that would make the theorem-level route competitive.

\paragraph{What the experiments show and do not show.}
The empirical conclusions are specific to the tested operators, problems, and initial populations. Constants estimated from calibration populations describe a family of populations as evidence, not as certificates. The study's recurring finding is that a component's value depends on the target: noise and recombination can raise the mean objective gap and still improve the best candidate found, and a population can agree at a point that is not optimal. A single averaged coefficient therefore does not rank components, and the choice of target (a retained solution, reliable current candidates, or agreement at an optimizer) should precede the choice of operators.

\paragraph{Code availability.}
The EvoScope toolbox, \url{https://github.com/pma-papers/evoscope}, implements the computations of this paper for user-defined operators and algorithms: numerical checks of the conditions of Assumption~\ref{assumption:composite}, the attribution of performance to operators, algorithm-level checks, and the measured discovery bound, with documentation and tutorials. The code, protocols, and results that reproduce every figure and table are at \url{https://github.com/pma-papers/operator-calculus-reproduction} (Online Appendix~S7).

\section*{Acknowledgments}
The underlying theoretical framework and original research ideas were developed by the authors. Generative AI tools assisted with technical proofreading, exposition, notation, and \LaTeX{} editing, and with developing and checking the worked examples and numerical checks. The authors are responsible for verifying the mathematical arguments, proofs, code, and all claims.

\appendix
\renewcommand{\theHsection}{app\Alph{section}}

\section{Notation and Preliminaries}\label{app:prelim}

This appendix collects the notation and background material used in the proofs. The conventions agree with Sections~\ref{sec:setup}--\ref{sec:operators}. Whenever differential tests on a proper subset $\X\subset\R^{d_x}$ are used, they are understood as restrictions of ambient smooth functions, with the stated derivative bounds valid on the segments used in Taylor and Lipschitz estimates. The transition kernels must also preserve $\X$. For $\X=\R^{d_x}$ no extension convention is needed.

\textbf{Measure spaces.}
Throughout, $\X\subseteq\R^{d_x}$ is a Polish subspace (typically $\X=\R^{d_x}$); the search space is $\S\subseteq\R^d$ where $d_x$ may differ from $d$. We write $\M(\X)$ for the space of finite signed Borel measures on $(\X,\mathscr{B}(\X))$, $\M_+(\X)$ for nonnegative finite measures, and $\P(\X)$ for probability measures. Since $\X$ is Polish, every $\mu\in\M(\X)$ is Radon. The total variation of $\mu\in\M(\X)$ is $|\mu|:=\mu^++\mu^-$, where $\mu^+,\mu^-$ are the positive and negative parts. For an exponent $q\ge 2$, the weight function is $w_q(x):=1+\|x\|^q$, and the cone of nonnegative measures with finite $q$-th moment and strictly positive total mass is
\[
\M_{q,+}(\X):=\bigl\{\mu\in\M_+(\X):\;\textstyle\int_\X w_q(x)\,\mu(dx)<\infty,\ \mu(\X)>0\bigr\}.
\] Similarly, $\P_q(\X)$ denotes the space of Borel probability measures on $\X$ with finite $q$-th moment, which becomes a metric space when endowed with the $p$-Wasserstein distance for $p\le q$~\citep{ambrosioGradientFlowsMetric2005}. The moment sublevel is $\M_{q,+}^r:=\{\mu\in\M_{q,+}(\X):\int w_q\,d\mu\le r\}$, and, when a uniform lower bound on total mass is required, the mass-restricted sublevel
\[
\M_{q,+}^{r,m}\;:=\;\bigl\{\mu\in\M_{q,+}(\X):\;\mu(\X)\ge m,\;\textstyle\int_\X w_q\,d\mu\le r\bigr\},\qquad m>0.
\]
The normalization is $\bar\mu:=\mu/\mu(\X)\in\P_q(\X)$ whenever $\mu\in\M_{q,+}(\X)$. The pairing of a measurable function $\varphi$ with a measure $\mu$ is $\langle\varphi,\mu\rangle:=\int\varphi\,d\mu$.

\textbf{Test function spaces and derivatives.}
For $k\in\N$, let $C_b^k(\X)$ be the Banach space of $k$-times continuously differentiable functions with bounded derivatives up to order $k$, endowed with the $C^k$-norm $\|\varphi\|_{C_b^k}:=\sum_{|\alpha|\le k}\|\partial^\alpha\varphi\|_\infty$, where $\|\partial^\alpha\varphi\|_\infty:=\sup_{x\in\X}|\partial^\alpha\varphi(x)|$. Unless otherwise stated, the test function space is $\mathcal{D}:=C_b^3(\X)$, and $\mathcal D'$ denotes its topological dual endowed with the weak topology $\sigma(\mathcal D',\mathcal D)$; the operator norm on $\mathcal D'$ is denoted $\|\cdot\|_{\mathcal D'}$. Matrix norms are the Frobenius norm $\|\cdot\|_F$ or the operator norm; the derivative-norm conventions of the operator verifications are stated in Online Appendix~S3.

\textbf{Bounded-Lipschitz norms and the weighted BL metric.}
Let $\mathrm{BL}(\X)\subset C_b(\X)$ denote the space of bounded Lipschitz functions with norm $\|\varphi\|_{\mathrm{BL}}:=\|\varphi\|_\infty+|\varphi|_L$, where $|\varphi|_L:=\sup_{x\ne y}|\varphi(x)-\varphi(y)|/\|x-y\|$. The dual bounded-Lipschitz (Dudley) norm on $\M(\X)$ is $\|\mu\|_{\mathrm{BL}}^*:=\sup\{\langle\varphi,\mu\rangle:\|\varphi\|_{\mathrm{BL}}\le1\}$, which metrizes weak convergence on $\M_+(\X)$ when $\X$ is Polish.

To handle polynomial growth, the weighted BL norm for $q\ge2$ (used in~\eqref{eq:dBLq}) is $\|f\|_{\mathrm{BL},q}:=\|f\|_{0,q}+[f]_{1,q-1}$, where
\[
\|f\|_{0,q}:=\sup_{x\in\X}\frac{|f(x)|}{w_q(x)}, \qquad [f]_{1,q-1}:=\sup_{x\ne y}\frac{|f(x)-f(y)|}{\|x-y\|(1+\|x\|^{q-1}+\|y\|^{q-1})},
\]
so $f$ grows at most polynomially of order $q$ with Lipschitz modulus of order $q-1$. The weighted BL distance~\eqref{eq:dBLq} on $\M_{q,+}(\X)$ is $d_{\mathrm{BL},q}(\mu,\nu):=\sup\{|\langle f,\mu-\nu\rangle|:\|f\|_{\mathrm{BL},q}\le1\}$. On any subset $\mathcal K\subset\M_{q,+}(\X)$ with $\sup_{\mu\in\mathcal K}\int\|x\|^q\,\mu(dx)<\infty$ and $\sup_{\mu\in\mathcal K}\mu(\X)<\infty$, $d_{\mathrm{BL},q}$ metrizes weak convergence together with convergence of the $q$-th moment~\citep{pflug2011approximations}. Since $\varphi\in\mathcal{D}$ is bounded and globally Lipschitz with $\mathrm{Lip}(\varphi)\le\|\nabla\varphi\|_\infty$, we have $\varphi\in\mathrm{BL}_q(\X)$ and
\[
|\langle\varphi,\mu-\nu\rangle|\le\|\varphi\|_{\mathrm{BL},q}\,d_{\mathrm{BL},q}(\mu,\nu),\qquad \|\varphi\|_{\mathrm{BL},q}\le\|\varphi\|_\infty+\|\nabla\varphi\|_\infty.
\]

\textbf{Wasserstein distances.}
For $p\ge 1$ and $\mu,\nu\in\P_p(\X)$, the \emph{$p$-Wasserstein distance} is
\begin{equation}\label{eq:wasserstein-app}
W_p(\mu,\nu):=\inf_{\gamma\in\Gamma(\mu,\nu)}\!\left(\int_{\X\times\X}\|x-y\|^p\,\gamma(dx,dy)\right)^{\!1/p},
\end{equation}
where $\Gamma(\mu,\nu)$ is the set of couplings of $\mu$ and $\nu$ (joint distributions with the prescribed marginals). Intuitively, $W_p$ measures the minimal transport cost of rearranging the mass of $\mu$ into the shape of $\nu$. The $p$-Wasserstein distance is monotone in $p$: for $1\le p\le q$ and $\mu,\nu\in\P_q(\X)$,
\begin{equation}\label{eq:wasserstein-distance-monotone}
W_p(\mu,\nu)\;\le\;W_q(\mu,\nu),
\end{equation}
by Jensen's inequality applied to the optimal coupling. For a point mass, $W_p^p(\nu,\delta_{y_*})=\int\|y-y_*\|^p\,d\nu$, so consensus in $W_p$ is equivalent to concentration of the $p$-th moment around $y_*$; the case $p=2$ is our default.

\textbf{Population metric and its role.}
To control evolutionary operators on $\M_{q,+}(\X)$ we need to simultaneously control linear observables $\mu\mapsto\int\varphi\,d\mu$ for polynomially growing $\varphi$, and nonlinear dependence on the normalized population law $\bar\mu=\mu/\mu(\X)$ that appears in mean-field coefficients (consensus points, pair laws, rank-based selection, law-dependent drifts and diffusions). The weighted BL metric $d_{\mathrm{BL},q}$ is convenient with test functions $\mathcal D$ because it directly bounds test-function increments; stability conditions for law-dependent coefficients, in contrast, are naturally expressed in Wasserstein distance---typically $W_2(\bar\mu,\bar\nu)$, which is the canonical metric for Lipschitz dependence on the law in McKean--Vlasov-type dynamics. Accordingly, we endow $\M_{q,+}(\X)$ with the \emph{population metric}
\begin{equation}\label{eq:pop-metric-app}
d_{q,2}^*(\mu,\nu)\;:=\;d_{\mathrm{BL},q}(\mu,\nu)+W_2(\bar\mu,\bar\nu),\qquad \bar\mu:=\mu/\mu(\X),\;\bar\nu:=\nu/\nu(\X),
\end{equation}
which is well defined whenever $q\ge 2$ (so that $\bar\mu,\bar\nu\in\P_2(\X)$) and $\mu,\nu\in\M_{q,+}(\X)$ (so that both total masses are positive).

\section{General Calculus and Lyapunov Proofs}\label{app:operator-calculus}\label{app:lyapunov-testing}

Notation is that of Appendix~\ref{app:prelim}; the canonical formulas are verified in Online Appendix~S3. Theorem~\ref{thm:operator-composition}, Corollary~\ref{cor:induced-S}, and Proposition~\ref{prop:operatorwise} are proved in Sections~\ref{sec:operators}--\ref{sec:lyapunov}; this appendix proves Theorem~\ref{thm:mean-field-convergence}.

\subsection{Pre-generator: formal definition}\label{app:pregenerators}

\begin{definition}[Pre-generator]\label{def:pregenerator}
For operators $T_\tau\colon\M_{q,+}(\X)\to\M_{q,+}(\X)$ and $\mathcal D=C_b^3(\X)$, a map $G\colon\M_{q,+}(\X)\to\mathcal D'$ is a pre-generator if, for every $\varphi\in\mathcal D$, $r<\infty$, and $m_0>0$,
\begin{equation}\label{eq:pregenerator-def}
\sup_{\mu\in\M_{q,+}^{r,m_0}}\left|
\frac{\langle\varphi,T_\tau[\mu]-\mu\rangle}{\tau}-G[\mu](\varphi)
\right|\longrightarrow0\qquad(\tau\downarrow0).
\end{equation}
\end{definition}
The mass floor $m_0$ is needed by the normalized Wasserstein part of the near-identity estimate: $\mu_R=(1+R^q)^{-1}\delta_R$ has bounded weighted moment, yet under $x\mapsto(1-\tau)x$ its normalized law moves by $W_2=\tau R$.

\subsection{Search-space compatibility}\label{app:lyap-compatibility}

\begin{lemma}[Search-space compatibility]\label{lem:search-space-compatibility}
Let $\bar\mu\in\P_q(\X)$, let $\nu:=\Pi[\bar\mu]\in\mathscr P(\S)$, and let $\Psi\colon\S\to[0,\infty]$ be measurable with $\Psi\in L^1(\nu)$. Then
\begin{equation}\label{eq:app-compat-identity}
\mathcal E_\Psi(\bar\mu)=\int_\S \Psi(y)\,\nu(dy)=\int_\X \ell_\Psi(x)\,\bar\mu(dx).
\end{equation}
Moreover:
\begin{enumerate}[label=\textup{(\roman*)},leftmargin=*,itemsep=2pt]
\item If $f_*>-\infty$ and $\Psi(y)=f(y)-f_*$, then $\mathcal E_\Psi(\bar\mu)=\int_\S f(y)\,\nu(dy)-f_*$.
\item If $\S_*=\{y_*\}$, $\Psi(y)=\|y-y_*\|^2$, and $\mathcal E_\Psi(\bar\mu)<\infty$, then $\mathcal E_\Psi(\bar\mu)=W_2^2(\nu,\delta_{y_*})$.
\item If, for some $\varepsilon>0$, there exists $c_\varepsilon>0$ such that $\Psi(y)\ge c_\varepsilon$ for all $y\in\S\setminus\S_*^\varepsilon$, then
$1-\nu(\S_*^\varepsilon)=\nu(\S\setminus\S_*^\varepsilon)\le c_\varepsilon^{-1}\mathcal E_\Psi(\bar\mu)$.
In particular, if this holds for every $\varepsilon>0$ and $\mathcal E_\Psi(\bar\mu_t)\to 0$, then the induced search laws achieve $\varepsilon$-concentration.
\end{enumerate}
Consequently, if $\ell_\Psi(x)\le C_K\,\Upsilon(x)$ for all $x\in\X$, then $\mathcal E_\Psi(\bar\mu)\le C_K\,\mathcal V_\Upsilon(\bar\mu)$ for every $\bar\mu\in\P_q(\X)$.
\end{lemma}

\begin{proof}
Tonelli gives $\int\Psi\,d\Pi[\bar\mu]=\int(\int\Psi(y)K(x,dy))\bar\mu(dx)=\int\ell_\Psi\,d\bar\mu$, which is~\eqref{eq:app-compat-identity}. (i) follows from normalization. For (ii), the only coupling of $\nu$ with $\delta_{y_*}$ is $\nu\otimes\delta_{y_*}$, with cost $\int\|y-y_*\|^2\,d\nu$; the identity also holds for a selected minimizer when the minimizer is not unique. For (iii), integrate $\Psi\ge c_\varepsilon\mathbf1_{\S\setminus\S_*^\varepsilon}$. Integrating $\ell_\Psi\le C_K\Upsilon$ gives the last claim.
\end{proof}

\subsection{The Lyapunov identity for admissible unbounded tests}\label{app:lyap-testing}
For a cutoff-admissible Lyapunov function (Definition~\ref{def:cutoff-admissible}), the weak evolution equation extends from bounded tests to $\Upsilon$ itself.

\begin{lemma}[Extension to an unbounded Lyapunov test]\label{lem:lyapunov-testing}
Let $(\bar\mu_t)_{t\ge0}\subset\P_q(\X)$ be a probability-valued $\mathcal D$-weak solution of~\eqref{eq:general-evolution}, supported in $\mathcal R$, with $q$-th moments bounded on every finite time interval. If $\Upsilon$ is cutoff-admissible as in Definition~\ref{def:cutoff-admissible}, then $t\mapsto\mathcal V_\Upsilon(\bar\mu_t)$ is locally absolutely continuous and, for $0\le s\le t$,
\begin{equation}\label{eq:generic-lyapunov-identity}
\mathcal V_\Upsilon(\bar\mu_t)-\mathcal V_\Upsilon(\bar\mu_s)
=\int_s^tG[\bar\mu_u](\Upsilon)\,du.
\end{equation}
In particular, its almost-everywhere derivative is $G[\bar\mu_t](\Upsilon)$.
\end{lemma}

\begin{proof}
Fix $T<\infty$ and put $M_T:=\sup_{0\le u\le T}\int w_q\,d\bar\mu_u<\infty$. Testing the weak equation with $\Upsilon_r\in\mathcal D$ gives, for $0\le s\le t\le T$,
$\langle\Upsilon_r,\bar\mu_t\rangle-\langle\Upsilon_r,\bar\mu_s\rangle=\int_s^t\sum_{j=1}^JG_j[\bar\mu_u](\Upsilon_r)\,du$.
Since $0\le\Upsilon_r\le\Upsilon\le C_0w_q$ and $\Upsilon_r\to\Upsilon$ pointwise, dominated convergence handles both endpoint integrals. By cutoff admissibility and finiteness of $J$, the integrand converges pointwise to $G[\bar\mu_u](\Upsilon)$ and is bounded by $\sum_jC_j(M_T)$, integrable on $[0,T]$; the limit is measurable as a pointwise limit along integer $r$ of measurable integrands. Dominated convergence in time proves~\eqref{eq:generic-lyapunov-identity}, whose right side is the integral of an $L^1(0,T)$ function, giving absolute continuity and the almost-everywhere derivative on $[0,T]$; $T$ was arbitrary.
\end{proof}

\subsection{Lyapunov decay}\label{app:lyap-decay}

\begin{proof}[Proof of the decay conclusion in Theorem~\ref{thm:mean-field-convergence}]
By Lemma~\ref{lem:lyapunov-testing} and the closed dissipation inequality, $\frac{d}{dt}\mathcal{V}_\Upsilon(\bar\mu_t)=G[\bar\mu_t](\Upsilon)\le -\lambda\mathcal{V}_\Upsilon(\bar\mu_t)$ for a.e.\ $t\ge0$, and Gr\"onwall's lemma gives $\mathcal{V}_\Upsilon(\bar\mu_t)\le e^{-\lambda t}\mathcal{V}_\Upsilon(\bar\mu_0)$. Lemma~\ref{lem:search-space-compatibility} applied to the objective-gap and quadratic gauges gives the transfer estimates
\begin{equation}\label{eq:app-transfer-estimates}
\begin{aligned}
\mathcal{E}_{\Psi_f}(\bar\mu_t)&\le C_{\mathrm{obj}}\,\mathcal{V}_\Upsilon(\bar\mu_t)\le C_{\mathrm{obj}}\,e^{-\lambda t}\mathcal{V}_\Upsilon(\bar\mu_0),\\
\mathcal{E}_{\Psi_*}(\bar\mu_t)&\le C_{\mathrm{quad}}\,\mathcal{V}_\Upsilon(\bar\mu_t)\le C_{\mathrm{quad}}\,e^{-\lambda t}\mathcal{V}_\Upsilon(\bar\mu_0),
\end{aligned}
\end{equation}
each under its own compatibility hypothesis and used below only for the corresponding conclusion.
\end{proof}

\begin{remark}[Decay with a residual]\label{rem:residual-decay}
If $G[\bar\mu](\Upsilon)\le-\lambda\mathcal V_\Upsilon(\bar\mu)+c$ holds only for $\bar\mu$ in a class $\mathcal A$ (possibly defined by integral constraints rather than a support set), Lemma~\ref{lem:lyapunov-testing} still gives $\frac{d}{dt}\mathcal V_\Upsilon(\bar\mu_t)=G[\bar\mu_t](\Upsilon)$, so as long as $\bar\mu_u\in\mathcal A$ on $[0,t]$, Gr\"onwall's lemma yields $\mathcal V_\Upsilon(\bar\mu_t)\le e^{-\lambda t}\mathcal V_\Upsilon(\bar\mu_0)+c/\lambda$, and~\eqref{eq:app-transfer-estimates} and parts (a)--(c) of Theorem~\ref{thm:mean-field-convergence} hold on $[0,t]$ with $e^{-\lambda t}\mathcal V_\Upsilon(\bar\mu_0)$ replaced by this quantity: a floor of order $c/\lambda$. Persistence of the class is a separate hypothesis; the CBO certificate of Online Appendix~S4.6 is used in this form.
\end{remark}

\subsection{Convergence modes: concentration and objective error}\label{app:lyap-modes}

\begin{proof}[Proof of Theorem~\ref{thm:mean-field-convergence}\ref{item:mf1} ($\varepsilon$-concentration)]
No basin-invariance hypothesis is used. Recall $r_\varepsilon=\min\{R_0,\varepsilon/L_0\}$ with $L_0=L_f(1+2\|y_*\|+2R_0)^s$. For $y\in B_{r_\varepsilon}(y_*)\cap\S$, Assumption~\ref{assumption:landscape}(i) and $f(y_*)=f_*$ give $f(y)-f_*\le L_f(1+\|y\|+\|y_*\|)^s\|y-y_*\|$, and since $\|y\|<\|y_*\|+R_0$ the right side is at most $L_0\|y-y_*\|<L_0r_\varepsilon\le\varepsilon$. Thus $\S\setminus\S_*^\varepsilon\subset\S\setminus B_{r_\varepsilon}(y_*)$, and Markov's inequality with~\eqref{eq:app-compat-identity} gives
\[
1-p_\varepsilon(t)\le\nu_t\{y:\|y-y_*\|^2\ge r_\varepsilon^2\}
\le r_\varepsilon^{-2}\int_\S\|y-y_*\|^2\,\nu_t(dy)
=r_\varepsilon^{-2}\mathcal E_{\Psi_*}(\bar\mu_t),
\]
and the quadratic estimate in~\eqref{eq:app-transfer-estimates} yields $1-p_\varepsilon(t)\le C_{\mathrm{quad}}\,r_\varepsilon^{-2}\,e^{-\lambda t}\mathcal V_\Upsilon(\bar\mu_0)$.
\end{proof}

\begin{proof}[Proof of Theorem~\ref{thm:mean-field-convergence}\ref{item:mf2} (objective-gap decay)]
Lemma~\ref{lem:search-space-compatibility}(i) with $\Psi(y)=f(y)-f_*$ gives $\int_\S f\,d\nu_t-f_*=\mathcal E_{\Psi_f}(\bar\mu_t)$, and the objective estimate in~\eqref{eq:app-transfer-estimates} yields the claimed bound. Since $\S\setminus\S_*^\varepsilon=\{y:f(y)-f_*>\varepsilon\}$, Markov's inequality also gives
$1-p_\varepsilon(t)\le\varepsilon^{-1}\mathcal E_{\Psi_f}(\bar\mu_t)\le C_{\mathrm{obj}}\,\varepsilon^{-1}e^{-\lambda t}\mathcal V_\Upsilon(\bar\mu_0)$,
without the landscape or basin assumptions.
\end{proof}

\begin{proof}[Proof of Theorem~\ref{thm:mean-field-convergence}\ref{item:mf3} (quadratic error and consensus)]
For any selected minimizer $y_*$ and $\Psi_*(y)=\|y-y_*\|^2$, the unique coupling with $\delta_{y_*}$ gives $\mathcal E_{\Psi_*}(\bar\mu_t)=W_2^2(\nu_t,\delta_{y_*})$, as in Lemma~\ref{lem:search-space-compatibility}(ii); with~\eqref{eq:app-transfer-estimates}, $W_2^2(\nu_t,\delta_{y_*})\le C_{\mathrm{quad}}e^{-\lambda t}\mathcal V_\Upsilon(\bar\mu_0)$. When $\S_*=\{y_*\}$, this is MF3 as defined in Section~\ref{sec:setup}.
\end{proof}

\section{Finite-Population Guarantees: Proofs}\label{app:finite-population}

\subsection{Finite-budget assumptions and discovery time}
\label{app:fp-setup}

Throughout, $\S$ is Polish and $(t_k,N_k)_{k\ge1}$ is an evaluation schedule with $0<t_1<t_2<\cdots$, $t_k\to\infty$, $N_k\in\N$. For each particle number $N$, the normalized particle state law at $t_k$ is $\bar\mu_{t_k}^N=\sum_{i=1}^Nw_{i,k}^N\delta_{X_{t_k}^{N,i}}\in\P_q(\X)$ with weights $w_{i,k}^N\ge0$, $\sum_iw_{i,k}^N=1$ ($w_{i,k}^N=N^{-1}$ when unweighted), $\nu_{t_k}^N:=\Pi[\bar\mu_{t_k}^N]\in\P(\S)$ is the induced search law, and $N_k$ is the size of the fresh candidate batch evaluated at $t_k$, not necessarily equal to $N$.

\paragraph{Fresh-batch convention.}
For each $N$, $(\Omega^N,\mathcal F^N,\Prob^N)$ carries a filtration $(\mathscr H_k^N)_{k\ge0}$, the random measures $(\bar\mu_{t_k}^N,\nu_{t_k}^N)_{k\ge1}$, and the candidate batches $\{Y_{k,j}^N:\ k\ge1,\ 1\le j\le N_k\}$, such that $\nu_{t_k}^N$ is $\mathscr H_{k-1}^N$-measurable, $\{(\bar\mu_{t_\ell}^N,\nu_{t_\ell}^N),Y_{\ell,j}^N:\ell\le k\}$ are $\mathscr H_k^N$-measurable, and, conditional on $\mathscr H_{k-1}^N$, $Y_{k,1}^N,\ldots,Y_{k,N_k}^N$ are i.i.d.\ with law $\nu_{t_k}^N$. $\E^N$ is expectation under $\Prob^N$.

Theorem~\ref{thm:mean-field-convergence}\ref{item:mf1} gives, for every $\varepsilon>0$,
\begin{equation}\label{eq:fp-mean-field-bound}
1-p_\varepsilon(t)\le K_\varepsilon e^{-\lambda t},
\qquad t\ge0,
\qquad
K_\varepsilon
=\frac{C_{\mathrm{quad}}}{r_\varepsilon^2}
\V_\Upsilon(\bar\mu_0),
\end{equation}
where $p_\varepsilon(t)=\nu_t(\S_*^\varepsilon)$ is the mean-field one-sample success probability of Section~\ref{sec:setup}. For fixed $\varepsilon>0$, the burn-in time and index are $T_\varepsilon:=\lambda^{-1}\log(\max\{1,2K_\varepsilon\})$ and $k_\varepsilon:=\min\{k\in\N:t_k\ge T_\varepsilon\}$, finite because $t_k\to\infty$, so that $p_\varepsilon(t_k)\ge\frac12$ for $k\ge k_\varepsilon$. For $T\ge t_{k_\varepsilon}$, write $m(T):=\#\{k:t_k\le T\}$, $A_\varepsilon:=\S_*^\varepsilon$, $c_\varepsilon:=f_*+\varepsilon$, and $\tau^{(N)}_\varepsilon:=\inf\{t_k:\ Y^N_{k,j}\in\S_*^\varepsilon\text{ for some }1\le j\le N_k\}$, $\inf\varnothing:=+\infty$, for the finite-particle $\varepsilon$-discovery time.

\begin{assumption}[Sampling-kernel $W_2$ regularity]
\label{ass:fp-kernel}
For every $x\in\X$, $K(x,\cdot)\in\P_2(\S)$, and there is
$L_K\in(0,\infty)$ such that
\begin{equation}\label{eq:fp-kernel-pointwise}
W_2(K(x,\cdot),K(x',\cdot))\le L_K\|x-x'\|,
\qquad x,x'\in\X.
\end{equation}
\end{assumption}

\begin{assumption}[Boundary-strip regularity]
\label{ass:fp-boundary}
Fix $\varepsilon>0$ and $T\ge t_{k_\varepsilon}$. Assume:
\begin{enumerate}[label=\textup{(\roman*)},leftmargin=*,itemsep=2pt]
\item for some
$r_0,\eta_{\rm lev},L_{T,\varepsilon},M_{T,\varepsilon}>0$,
$f$ is $L_{T,\varepsilon}$-Lipschitz on
$\mathcal N_{r_0}(\partial A_\varepsilon)$;
\item for every
$k_\varepsilon\le k\le m(T)$, the pushforward $f_\#\nu_{t_k}$ has a
density $g_{k,\varepsilon}$ on
$[c_\varepsilon-\eta_{\rm lev},c_\varepsilon+\eta_{\rm lev}]$ with
$\|g_{k,\varepsilon}\|_{L^\infty([c_\varepsilon-\eta_{\rm lev},c_\varepsilon+\eta_{\rm lev}])}\le M_{T,\varepsilon}$;
\item $L_{T,\varepsilon}r_0\le\eta_{\rm lev}$;
\item either $A_\varepsilon=\S$ or, for every $r\in(0,r_0]$,
$\{y\in A_\varepsilon:\dist(y,A_\varepsilon^c)\le r\}\subset\mathcal N_r(\partial A_\varepsilon)$.
\end{enumerate}
Here $\mathcal N_r(B):=\{y\in\S:\dist(y,B)\le r\}$.
\end{assumption}

\subsection{Stability of the induced search law}
\label{app:fp-set-mass}

\begin{lemma}[Stability of the induced search law]
\label{lem:fp-Pi-stability}
Under Assumption~\ref{ass:fp-kernel}, for all
$\alpha,\beta\in\P_2(\X)$,
\begin{equation}\label{eq:fp-Pi-stability}
W_2(\Pi[\alpha],\Pi[\beta])\le L_KW_2(\alpha,\beta).
\end{equation}
\end{lemma}

\begin{proof}
Let $\gamma\in\Gamma(\alpha,\beta)$ be optimal for $W_2(\alpha,\beta)$ and, by measurable selection of optimal plans on Polish spaces, let $\gamma_{x,x'}$ be an optimal coupling of $K(x,\cdot)$ and $K(x',\cdot)$ chosen measurably in $(x,x')$. Then $\Gamma:=\int\gamma_{x,x'}\,d\gamma(x,x')$ couples $\Pi[\alpha]$ and $\Pi[\beta]$, so by Assumption~\ref{ass:fp-kernel}
\[
W_2^2\bigl(\Pi[\alpha],\Pi[\beta]\bigr)
\le\int_{\S\times\S}\|y-y'\|^2\,d\Gamma
=\int_{\X\times\X}
W_2^2\bigl(K(x,\cdot),K(x',\cdot)\bigr)\,d\gamma(x,x')
\le L_K^2W_2^2(\alpha,\beta).
\]
\end{proof}

\subsection{Finite-particle discovery bound}
\label{app:fp-discovery}

\begin{proposition}[Finite-particle discovery bound]
\label{thm:discovery-app}
Under the fresh-batch convention of Appendix~\ref{app:fp-setup}, for
$m\in\N$ define
$B_{N,m}(\varepsilon):=\sum_{k=1}^m\Prob^N\bigl(\nu^N_{t_k}(\S_*^\varepsilon)<\tfrac12\,p_\varepsilon(t_k)\bigr)$.
Then, for every $m\in\N$,
\begin{equation}\label{eq:discovery-abstract-app}
\Prob^N\!\bigl(\tau^{(N)}_\varepsilon>t_m\bigr)
\le
B_{N,m}(\varepsilon)
+\exp\!\Bigl(-\tfrac12\sum_{k=1}^mN_k\,p_\varepsilon(t_k)\Bigr).
\end{equation}
Moreover, Theorem~\ref{thm:mean-field-convergence}\ref{item:mf1} gives
\begin{equation}\label{eq:discovery-explicit-app}
\Prob^N\!\bigl(\tau^{(N)}_\varepsilon>t_m\bigr)
\le
B_{N,m}(\varepsilon)
+\exp\!\Bigl(
-\tfrac12\sum_{k=1}^mN_k
\bigl[1-K_\varepsilon e^{-\lambda t_k}\bigr]_+
\Bigr).
\end{equation}
\end{proposition}

\begin{proof}
Write $p_k=p_\varepsilon(t_k)$, $E_k=\{\nu^N_{t_k}(\S_*^\varepsilon)\ge p_k/2\}$, $F_k$ for the event that all candidates in batch $k$ miss $\S_*^\varepsilon$, $H_k=\bigcap_{\ell\le k}F_\ell=\{\tau_\varepsilon^{(N)}>t_k\}$ and $H_0=\Omega^N$. $E_k$ and $H_{k-1}$ are $\mathscr H_{k-1}^N$-measurable, and conditional independence within the fresh batch gives $\Prob^N(F_k\mid\mathscr H_{k-1}^N)=(1-\nu^N_{t_k}(\S_*^\varepsilon))^{N_k}$. Splitting $H_k$ according to $E_k$ and conditioning the first part on $\mathscr H_{k-1}^N$, with $\rho_k=(1-p_k/2)^{N_k}$,
\begin{equation}\label{eq:discovery-recursion-app}
\Prob^N(H_k)\le\Prob^N(E_k^c)+\E^N[\mathbf1_{H_{k-1}}\mathbf1_{E_k}\Prob^N(F_k\mid\mathscr H_{k-1}^N)]
\le\Prob^N(E_k^c)+\rho_k\Prob^N(H_{k-1}).
\end{equation}
Repeated substitution from $\Prob^N(H_0)=1$ yields, without independence across batches,
\begin{equation}\label{eq:discovery-unrolled-app}
\Prob^N(H_m)\le\sum_{k=1}^m\Prob^N(E_k^c)\prod_{\ell=k+1}^m\rho_\ell+\prod_{k=1}^m\rho_k.
\end{equation}
Since $\rho_k\le1$ and $1-u\le e^{-u}$, the right side is at most $B_{N,m}(\varepsilon)+\exp(-\frac12\sum_{k=1}^mN_kp_k)$, proving~\eqref{eq:discovery-abstract-app}; Theorem~\ref{thm:mean-field-convergence}\ref{item:mf1} gives $p_k\ge[1-K_\varepsilon e^{-\lambda t_k}]_+$, hence~\eqref{eq:discovery-explicit-app}.
\end{proof}

\subsection{Evaluation complexity at finite budget: proof of
Corollary~\ref{thm:complexity-main}}
\label{app:fp-complexity}

\paragraph{Initial support is not a finite-budget guarantee.}
The CBO condition $y_*\in\operatorname{supp}(\nu_0)$ gives $p_\varepsilon(0)>0$ when $f$ is continuous at $y_*$, but no useful lower bound: the probability that none of $N$ i.i.d.\ initial candidates is $\varepsilon$-optimal is $(1-p_\varepsilon(0))^N$, which can stay close to one. Corollary~\ref{thm:complexity-main} assumes neither initial support nor initial concentration: $1-p_\varepsilon(t)\le K_\varepsilon e^{-\lambda t}$ gives $p_\varepsilon(t_k)\ge1/2$ after $k_\varepsilon$, so on the adequate-mass event of Proposition~\ref{prop:fp-discounted} each fresh candidate succeeds with probability at least $1/4$. A dedicated CBO concentration result~\citep{fornasierConsensusBasedOptimizationMethods2024} can replace the mean-field input only if its error quantity and horizon match the bound used here and the batches are fresh draws rather than the already-evaluated interacting positions.

\subsubsection{The discounted discovery bound}
\label{app:fp-discounted-bound}

\begin{proposition}[Discounted post-burn-in discovery bound]
\label{prop:fp-discounted}
Suppose the fresh-batch convention of Appendix~\ref{app:fp-setup}
holds.
For $k\ge k_\varepsilon$, set
$a_{N,k}(\varepsilon):=\Prob^N\bigl(\nu^N_{t_k}(\S_*^\varepsilon)<\tfrac12p_\varepsilon(t_k)\bigr)$
and
$\rho_k(\varepsilon):=\bigl(1-\tfrac12p_\varepsilon(t_k)\bigr)^{N_k}$.
Here $a_{N,k}$ is the \emph{inadequate-mass probability}, the probability that the finite search law assigns too little mass to the success set (not a candidate-miss probability), and on the complementary adequate-mass event $\rho_k$, the \emph{survival factor}, bounds the conditional probability that all $N_k$ fresh candidates miss.
Then $p_\varepsilon(t_k)\ge\frac12$ for all
$k\ge k_\varepsilon$, and, for every $m\ge k_\varepsilon$,
\begin{equation}\label{eq:fp-discounted-exact}
\Prob^N\!\bigl(\tau_\varepsilon^{(N)}>t_m\bigr)
\le
\sum_{j=k_\varepsilon}^m a_{N,j}(\varepsilon)
\prod_{k=j+1}^m\rho_k(\varepsilon)
+\prod_{k=k_\varepsilon}^m\rho_k(\varepsilon),
\end{equation}
where an empty product equals $1$: the summand indexed by $j$ charges an inadequate-mass event at batch $j$ and discounts it by the survival factors of the later batches, and the final product is the residual pure-sampling term when every post-burn-in batch has adequate mass. In particular,
\begin{equation}\label{eq:fp-discounted-bound}
\Prob^N\!\bigl(\tau_\varepsilon^{(N)}>t_m\bigr)
\le
\widetilde B^{\rm disc}_{N,m}(\varepsilon)
+\exp\!\Bigl(-\tfrac14\sum_{k=k_\varepsilon}^mN_k\Bigr),
\end{equation}
where
\begin{equation}\label{eq:fp-discounted-def}
\widetilde B^{\rm disc}_{N,m}(\varepsilon)
:=
\sum_{j=k_\varepsilon}^m a_{N,j}(\varepsilon)
\exp\!\Bigl(-\tfrac14\sum_{k=j+1}^mN_k\Bigr).
\end{equation}
\end{proposition}

\begin{proof}
The definition of $T_\varepsilon$ and~\eqref{eq:fp-mean-field-bound} imply $p_\varepsilon(t_k)\ge1/2$ for $k\ge k_\varepsilon$. Starting recursion~\eqref{eq:discovery-recursion-app} at $k_\varepsilon$ with initial non-discovery probability at most one, repeated substitution gives~\eqref{eq:fp-discounted-exact}. Since $\rho_k\le(3/4)^{N_k}\le e^{-N_k/4}$ after burn-in, replacing each product by the corresponding exponential gives~\eqref{eq:fp-discounted-bound}.
\end{proof}

\subsubsection{The modulus and boundary regularity}
\label{app:fp-boundary-regularity}

The state-law approximation modulus $\mathcal A_N(T)$ of Corollary~\ref{thm:complexity-main} is an algorithm-specific input; Proposition~S5.3 bounds it by a propagation-of-chaos coupling error plus an empirical-measure error. Wasserstein closeness alone does not control the mass of a closed set, since boundary mass can cross under an arbitrarily small perturbation; Assumption~\ref{ass:fp-boundary} rules this out, and Proposition~S5.4 gives sufficient conditions for both transfer assumptions.

\begin{proposition}[Boundary strip from objective-value densities]
\label{prop:fp-boundary-strip}
Under Assumption~\ref{ass:fp-boundary}, for every $r\in(0,r_0]$,
\begin{equation}\label{eq:fp-boundary-strip}
\sup_{k_\varepsilon\le k\le m(T)}
\nu_{t_k}\!\left(\mathcal N_r(\partial A_\varepsilon)\right)
\le2L_{T,\varepsilon}M_{T,\varepsilon}r.
\end{equation}
\end{proposition}

\begin{proof}
The empty-boundary case is immediate. Otherwise continuity gives $f=c_\varepsilon$ on $\partial A_\varepsilon$, and for $y\in\mathcal N_r(\partial A_\varepsilon)$ the local Lipschitz bound yields $|f(y)-c_\varepsilon|\le L_{T,\varepsilon}r$, so the strip lies in $f^{-1}(I_r)$, $I_r=[c_\varepsilon-L_{T,\varepsilon}r,c_\varepsilon+L_{T,\varepsilon}r]$. Since $L_{T,\varepsilon}r\le\eta_{\rm lev}$, the density bound gives, uniformly over the stated times, $\nu_{t_k}(\mathcal N_r(\partial A_\varepsilon))\le\int_{I_r}g_{k,\varepsilon}(u)\,du\le|I_r|M_{T,\varepsilon}=2L_{T,\varepsilon}M_{T,\varepsilon}r$.
\end{proof}

\subsubsection{Cutoff control and proof of Corollary~\ref{thm:complexity-main}}
\label{app:fp-proof}

\begin{corollary}[Cutoff control of the discounted approximation error]
\label{cor:fp-cutoff}
Suppose the fresh-batch convention and
Assumptions~\ref{ass:fp-kernel}--\ref{ass:fp-boundary} hold, and define
$r_*(T,\varepsilon):=r_0\wedge(16L_{T,\varepsilon}M_{T,\varepsilon})^{-1}$. Then
\begin{align}
\widetilde B^{\rm disc}_{N,m(T)}(\varepsilon)
&\le
\Bigl(\frac{8L_K}{r_*(T,\varepsilon)}\Bigr)^2\mathcal A_N(T)
\sum_{j=k_\varepsilon}^{m(T)}
\exp\!\Bigl(-\tfrac14\sum_{k=j+1}^{m(T)}N_k\Bigr).
\label{eq:fp-cutoff-general}
\end{align}
If $N_k=N$ after burn-in, then
\begin{equation}\label{eq:fp-cutoff-uniform}
\widetilde B^{\rm disc}_{N,m(T)}(\varepsilon)
\le\frac1{1-e^{-N/4}}
\Bigl(\frac{8L_K}{r_*(T,\varepsilon)}\Bigr)^2\mathcal A_N(T).
\end{equation}
\end{corollary}

\begin{proof}
Write $A=\S_*^\varepsilon$, $r_*=r_*(T,\varepsilon)$, and $p_j=p_\varepsilon(t_j)$. If $A=\S$, all $a_{N,j}$ vanish. Otherwise, Proposition~\ref{prop:fp-boundary-strip} and the definition of $r_*$ give $\nu_{t_j}(\mathcal N_{r_*}(\partial A))\le1/8$. The cutoff $\psi(y)=\min\{1,\dist(y,A^c)/r_*\}$ is $1/r_*$-Lipschitz and satisfies
$\mathbf1_{A\setminus\mathcal N_{r_*}(\partial A)}\le\psi\le\mathbf1_A$ by the geometric clause of Assumption~\ref{ass:fp-boundary}. Hence, by $W_1\le W_2$ and Lemma~\ref{lem:fp-Pi-stability},
\[
\nu^N_{t_j}(A)\ge\langle\psi,\nu^N_{t_j}\rangle
\ge p_j-\tfrac18-\frac{W_2(\nu^N_{t_j},\nu_{t_j})}{r_*}
\ge p_j-\tfrac18-\frac{L_K}{r_*}W_2(\bar\mu^N_{t_j},\bar\mu_{t_j}).
\]
For $j\ge k_\varepsilon$, $p_j\ge1/2$, so $W_2(\bar\mu^N_{t_j},\bar\mu_{t_j})\le r_*/(8L_K)$ implies $\nu^N_{t_j}(A)\ge p_j-1/4\ge p_j/2$, and Markov's inequality gives
\[
a_{N,j}(\varepsilon)\le\Prob^N\!\left(W_2(\bar\mu^N_{t_j},\bar\mu_{t_j})>\frac{r_*}{8L_K}\right)
\le\left(\frac{8L_K}{r_*}\right)^2\mathcal A_N(T).
\]
Substitution into~\eqref{eq:fp-discounted-def} proves~\eqref{eq:fp-cutoff-general}. If $N_k=N$, the remaining sum is $\sum_{\ell=0}^{m(T)-k_\varepsilon}e^{-N\ell/4}\le(1-e^{-N/4})^{-1}$, proving~\eqref{eq:fp-cutoff-uniform}.
\end{proof}

\begin{proof}[Proof of Corollary~\ref{thm:complexity-main}]
Clause (iv) of Assumption~\ref{ass:fp-boundary} is immediate if $A_\varepsilon=\S$. Otherwise let $y\in A_\varepsilon$ with $\dist(y,A_\varepsilon^c)\le r\le r_0$. For every $\zeta>0$, choose $z\in A_\varepsilon^c$ with $\|y-z\|<\dist(y,A_\varepsilon^c)+\zeta$ and, by the length-space property, a curve from $y$ to $z$ of length less than $\|y-z\|+\zeta$. The curve meets $\partial A_\varepsilon$, hence $\dist(y,\partial A_\varepsilon)<\dist(y,A_\varepsilon^c)+2\zeta\le r+2\zeta$, and $\zeta\downarrow0$ gives $y\in\mathcal N_r(\partial A_\varepsilon)$. Thus, with $T=t_m$, hypothesis~(iii) of the corollary gives Assumption~\ref{ass:fp-boundary}, and (i)--(ii) are the fresh-batch convention and Assumption~\ref{ass:fp-kernel}.

Proposition~\ref{prop:fp-discounted}, then Corollary~\ref{cor:fp-cutoff} with the definition~\eqref{eq:fp-modulus-hypothesis}, give
\[
\Prob^N(\tau_\varepsilon^{(N)}>t_m)
\le
\widetilde B^{\rm disc}_{N,m}(\varepsilon)
+
\exp\!\Bigl(-\tfrac14\sum_{k=k_\varepsilon}^mN_k\Bigr)
\le
\mathfrak A_N^{\rm disc}(T,\varepsilon)
+
\exp\!\Bigl(-\tfrac14\sum_{k=k_\varepsilon}^mN_k\Bigr)
\le\tfrac\delta2+\tfrac\delta2
\]
under \eqref{eq:fp-criterion}, the last step because $\exp(-\frac14\sum_{k=k_\varepsilon}^mN_k)\le\exp(-\log\frac2\delta)$. Hence $\Prob^N(\tau_\varepsilon^{(N)}\le t_m)\ge1-\delta$.
\end{proof}

\vskip 0.2in
\bibliography{references}

\end{document}